\documentclass[aop]{imsart}
\RequirePackage{amsthm,amsmath,amsfonts,amssymb}
\RequirePackage[authoryear]{natbib}%% uncomment this for author-year citations
\RequirePackage[colorlinks,citecolor=blue,urlcolor=blue]{hyperref}
\RequirePackage{graphicx}

\startlocaldefs
\theoremstyle{plain}

\newtheorem{theorem}{Theorem}[section]
\theoremstyle{remark}
\newtheorem{definition}[theorem]{Definition}
\newtheorem{example}{Example}

\usepackage{amsmath,amssymb,amsfonts,amsthm,bbm,empheq}
\usepackage{epic,eepic,epsfig,longtable}
\usepackage{multirow,multicol,verbatim}
\usepackage{array}
\usepackage{lscape}
\usepackage{graphicx}
\usepackage{latexsym}
\usepackage{comment}
\usepackage{booktabs}
\usepackage{epsfig}
\usepackage{CJK}
\usepackage{color}
\usepackage{mathtools}
\usepackage{extarrows}
\usepackage{centernot}
\usepackage{caption}
\usepackage{subcaption}
\usepackage{tikz}
\usetikzlibrary{shapes.multipart,shapes.geometric,spy}

\numberwithin{equation}{section}

\NewCommandCopy{\baccent}{\b}
\AtBeginDocument{%
  \DeclareRobustCommand{\b}[1]{\ifmmode\mathbf{#1}\else\baccent{#1}\fi}%
}

\def\[{\left [}  \def\]{\right ]} \def\({\left (}  \def\){\right )}

\makeatletter
\def\underbar#1{\underline{\sbox\tw@{$#1$}\dp\tw@\z@\box\tw@}}
\makeatother

\def\tilde{\widetilde}

\def\newpage{\vfill\eject}
\def\today{\ifcase\month\or
  January\or February\or March\or April\or May\or June\or
  July\or August\or September\or October\or November\or December\fi
  \space\number\day, \number\year}

\usepackage{stackengine} 
\stackMath

\setstackgap{L}{8pt}

\allowdisplaybreaks
\def\thesection{\arabic{section}}
\usepackage{verbatim}
\usepackage{lipsum}
\usepackage{enumitem}
\usepackage{hyperref}
\usepackage{xcolor}

\def\:={\coloneqq}
\def\=:{\eqqcolon}
\theoremstyle{plain}
\ifx\theorem\undefined
  \newtheorem{theorem}{Theorem}[section]
\fi
\ifx\proposition\undefined
  \newtheorem{proposition}{Proposition}[section]
\fi
\ifx\corollary\undefined
  \newtheorem{corollary}{Corollary}[section]
\fi
\ifx\lemma\undefined
  \newtheorem{lemma}{Lemma}[section]
\fi

\theoremstyle{definition}
\ifx\definition\undefined
  \newtheorem{definition}{Definition}[section]
\fi
\ifx\assumption\undefined
  \newtheorem{assumption}{Assumption}
\fi
\ifx\example\undefined
  \newtheorem{example}{Example}[section]
\fi
\ifx\exercise\undefined
  \newtheorem{exercise}{Exercise}[section]
\fi
\ifx\conjecture\undefined
  \newtheorem{conjecture}{Conjecture}[section]
\fi

\theoremstyle{remark}
\ifx\remark\undefined
  \newtheorem{remark}{Remark}
\fi

\usepackage[capitalise,nameinlink,noabbrev]{cleveref}
\crefformat{equation}{Eq.\ (#2#1#3)}
\crefrangeformat{equation}{Eqs.\ (#3#1#4) to~(#5#2#6)}
\crefmultiformat{equation}{Eq.\ (#2#1#3)}%
{ and~(#2#1#3)}{, (#2#1#3)}{ and~(#2#1#3)}
\crefname{assumption}{Assumption}{Assumptions}
\crefname{theorem}{Theorem}{Theorems}
\crefname{proposition}{Proposition}{Propositions}
\crefname{corollary}{Corollary}{Corollaries}
\crefname{lemma}{Lemma}{Lemmas}
\crefname{remark}{Remark}{Remarks}
\crefname{definition}{Definition}{Definitions}
\crefname{example}{Example}{Examples}

\newcommand{\reels}{\mathbb{R}}

\newcommand{\proba}{\mathbb{P}}
\newcommand{\Tau}{\mathrm{T}}
\endlocaldefs

\begin{document}

\begin{frontmatter}
\title{Explicit formula of boundary crossing probabilities for continuous local martingales to constant boundary}
\runtitle{Formula of boundary crossing probabilities}
%\thankstext{T1}{A sample additional note to the title.}

\begin{aug}
%%%%%%%%%%%%%%%%%%%%%%%%%%%%%%%%%%%%%%%%%%%%%%%
%% Only one address is permitted per author. %%
%% Only division, organization and e-mail is %%
%% included in the address.                  %%
%% Additional information can be included in %%

\author[A]{\fnms{Yoann}~\snm{Potiron}\ead[label=e1]{potiron@fbc.keio.ac.jp}},
%%%%%%%%%%%%%%%%%%%%%%%%%%%%%%%%%%%%%%%%%%%%%%
%% Addresses                                %%
%%%%%%%%%%%%%%%%%%%%%%%%%%%%%%%%%%%%%%%%%%%%%%
\address[A]{Faculty of Business and Commerce, Keio University\printead[presep={,\ }]{e1}}
\end{aug}

\begin{abstract}
An explicit formula for the probability that a continuous local martingale crosses a one or two-sided random constant boundary in a finite time interval is derived. We obtain that the boundary crossing probability of a continuous local martingale to a constant boundary is equal to the boundary crossing probability of a standard Wiener process to a constant boundary up to a time change of quadratic variation value. This relies on the constancy of the boundary and the Dambis, Dubins-Schwarz theorem for continuous local martingale. The main idea of the proof is the scale invariant property of the time-changed Wiener process and thus the scale invariant property of the first-passage time. As an application, we also consider an inverse first-passage time problem of quadratic variation.
\end{abstract}



\end{frontmatter}
%%%%%%%%%%%%%%%%%%%%%%%%%%%%%%%%%%%%%%%%%%%%%%
%% Please use \tableofcontents for articles %%
%% with 50 pages and more                   %%
%%%%%%%%%%%%%%%%%%%%%%%%%%%%%%%%%%%%%%%%%%%%%%
%\tableofcontents

\section{Introduction}
\subsection{Problem} 
Let $(Z_t)_{t \in \reels^+}$ be a continuous local martingale, while $g$ and $h$ are two random constant boundaries. We are concerned with one-sided and two-sided boundary crossing probabilities of the form
\begin{eqnarray}
\label{def_bcp}
P_g^Z(t) & = & \proba \big( \sup_{0 \leq s \leq t} Z_s    \geq g\big),\\ \label{def_bcp2} P_{g,h}^Z(t) & = & \proba \big( \sup_{0 \leq s \leq t} Z_s   \geq g \text{ or } \inf_{0 \leq s \leq t} Z_s  \leq h \big), 
\end{eqnarray}
i.e., the probability that the process $Z$ crosses the boundary or one of both boundaries between $0$ and $t$ for a finite time $t \geq 0$. 

As an application, we also consider the following inverse first-passage time (FPT) problem. For any known random survival cdf,
we want to determine the quadratic variation $\langle Z\rangle$ such that the boundary crossing probability (\ref{def_bcp})-(\ref{def_bcp2}) is equal to the survival cdf.

\subsection{Motivation}
The use of boundary crossing probabilities and FPT in statistics can at least be traced back to the one-sample Kolmogorov-Smirnov statistic for which the stochastic process is equal to the difference between the true and empirical cumulative distribution function (cdf). The main first 
application was in sequential analysis. There has been a first main focus when the stochastic process $Z$ is a random walk. Since the problem is harder to solve in that case, the literature relies on a continuous approximation and develop theoretical tools when the stochastic process $Z$ is a Wiener process (see \cite{gut1974moments}, \cite{woodroofe1976renewal}, \cite{woodroofe1977second}, \cite{lai1977nonlinear}, \cite{lai1979nonlinear} and \cite{siegmund1986boundary}). Another field of application is in survival analysis where \cite{matthews1985asymptotic} show that tests for constant hazard involve the FPT of an Ornstein-Uhlenbeck process and \cite{butler1997stochastic} present a Bayesian approach when the stochastic process is a semi-Markov process. \cite{eaton1977length} discuss the application of FPT for hospital stay. \cite{aalen2001understanding} study the case when the stochastic process is a Markov process. Another application is in pricing barrier options in mathematical finance (\cite{roberts1997pricing}). There are also some applications in econometrics where \cite{abbring2012mixed} and \cite{renault2014dynamic} study mixed first-hitting times where the stochastic process is respectively a spectrally negative Levy process and the sum of a Brownian motion and a positive linear drift. Detailed reviews on FPT are available in \cite{lee2006threshold} and \cite{lawless2011statistical} (Section 11.5, pp. 518-523).

In financial econometrics, FPT of a continuous local martingale are useful models of endogeneity when estimating the quadratic variation of the continuous local martingale. \cite{fukasawa2010central} considers the symmetric two-sided boundary case where $h = -g$. \cite{robert2011new} and \cite{robert2012volatility} (see also Section 4.4 in \cite{potiron2020local} for its extension) introduce the model with uncertainty zones where the two-sided boundary is dynamic. \cite{fukasawa2012central} consider the non-symmetric two-sided boundary case where $h \neq -g$. Finally, \cite{potiron2017estimation} estimate the quadratic covariation between two local martingales based on endogenous observations generated by FPT of an It\^{o}-semimartingale to a stochastic boundary process. Although these models are useful, the question remains in how general they are. With the use of the inverse FTP problem results, we can deduce the distributions allowed, and even associate to each distribution a quadratic variation of the stochastic process $Z$.

The main current applications when $Z$ is a continuous local martingale are pricing barrier options for boundary crossing probabilities and financial econometrics for inverse FPT problems. Most papers on boundary crossing probabilities have been done under the assumption that $Z$ is a Wiener process with constant variance, whereas we have empirical evidence that the variance is not constant in financial applications. In fact, it is common to assume under arbitrage theory that $Z$ follows a continuous local martingale process. Despite the growing importance of these results for applications, there is no explicit formula for general continuous local martingale in the literature.

\subsection{Existing results}
When $Z$ is a standard Wiener process with linear drift and the boundaries are linear, \cite{doob1949heuristic} gives explicit formulae (Equations (4.2)-(4.3), pp. 397-398) based on elementary geometrical and analytical arguments when $T=\infty$, $\sigma$ is nonrandom, the drift is null and the boundaries are nonrandom linear with nonnegative upper trend and nonpositive lower trend. \cite{malmquist1954certain} uses Doob's transformation (Section 5, pp. 401-402) to obtain an explicit formula conditioned on the starting and final values of the Wiener process (Theorem 1, p. 526) in the one-sided boundary case. \cite{anderson1960modification} derives an explicit formula conditioned on the final value of the Wiener process (Theorem 4.2, pp. 178-179) and integrate with respect to the final value of the Wiener process to get an explicit formula (Theorem 4.3, p. 180) in the two-sided boundary case with linear drift. \cite{potiron2023formula} gives several not fully explicit formulae for the probability that a Wiener process with stochastic drift process and random variance crosses a one-sided stochastic boundary process by the Girsanov theorem but the formula.

There are also some instances where formulae of the Laplace transform are obtained. For square root boundaries $g(t) = \sqrt{a+ t}$ with $a >0$ we can use Doob's transformation, and express Equations (\ref{def_bcp})-(\ref{def_bcp2}) as boundary crossing probabilities of an Ornstein-Unlenbeck process to a constant boundary (see \cite{breiman1967first}). However, the boundary crossing probabilities of an Ornstein-Unlenbeck process to a constant boundary are only known in the form of Laplace transform. \cite{daniels1969minimum} studies the probability that an Ornstein-Uhlenbeck process crosses a certain curved boundary, by equating it to the corresponding probability for a Wiener process to cross a transformed boundary which is chosen to yield an explicit formula. \cite{nobile1985exponential} investigate the asymptotic behaviour of the first-passage time (FPT) probability distribution function (pdf) to a large constant boundary by an Ornstein–Uhlenbeck process. The boundary crossing probabilities of a jump diffusion process with linear drift to a constant boundary are obtained in the form of Laplace transform (see \cite{kou2003first}). Finally, \cite{borovkov2008exit} find an explicit formula for the Laplace transform of the FPT of Levy-driven Ornstein–Uhlenbeck processes to a two-sided constant boundary under the assumption that positive jumps of the Levy process are exponentially distributed. 
 
Since there is no available explicit formula in general, there is a large literature on approximating and computing numerically these boundary crossing probabilities (\ref{def_bcp})-(\ref{def_bcp2}). When $Z$ is a standard Wiener process, \cite{strassen1967almost} (Lemma 3.3, p. 323) shows that $P_g^Z$ is continuous with continuous derivative when $g$ is continuous with continuous derivative. \cite{durbin1971boundary}, \cite{wang1997boundary} and \cite{novikov1999approximations} use piecewise-linear boundaries to approximate the general boundaries. \cite{durbin1985first} gives
a formula for a general boundary but which depends on asymptotic conditional expectations whose approximations are studied in
\cite{salminen1988first}. \cite{buonocore1987new}, \cite{giorno1989evaluation} and \cite{gutierrez1997first} consider the case when $Z$ is a general diffusion.

As for the inverse FPT problem, there is as far as the author knows no related paper on it, and this a completely new problem.

\subsection{This paper}
In this paper, we first derive an explicit formula for the one-sided and two-sided boundary crossing probability (\ref{def_bcp})-(\ref{def_bcp2}). We derive the results in two cases, i.e., (i) a simpler case when the boundary is nonrandom constant and the quadratic variation of the continuous local martingale is a time-varying function but not a stochastic process and (ii) a more complicated case when the boundary is random constant and the quadratic variation of the continuous local martingale is a stochastic process.

We describe first the boundary crossing probability main results in the one-sided (i) case. We obtain that the boundary crossing probability of a continuous local martingale to a constant boundary is equal to the boundary crossing probability of a standard Wiener process to a constant boundary up to a time change of quadratic variation value. More specifically, we obtain that (see Theorem \ref{theorem_nrc1})
$$P_g^Z (t)  =  P_g^W \big(\langle Z\rangle_{t}\big) \text{ for } t \geq 0,$$ 
where $W$ is defined as a standard Wiener process. This relies on the constancy of the boundary and the Dambis, Dubins-Schwarz theorem for continuous local martingale. The main idea of the proof is the scale invariant property of the time-changed Wiener process and thus the scale invariant property of the FPT. In the two-sided (i) case, the main result of this paper adapts since the arguments of boundary constancy and Dambis, Dubins-Schwarz theorem still hold. We obtain that (see Theorem \ref{theorem_nrc2}) 
$$P_{g,h}^Z (t)  =  P_{g,h}^W \big(\langle Z\rangle_{t}\big) \text{ for } t \geq 0.$$

To apply the Dambis, Dubins-Schwarz theorem in the one-sided (ii) case, the main idea is to rewrite the FPT to a random boundary as an equivalent FPT to a nonrandom boundary, i.e., by defining the new stochastic process as $Y = \frac{Z}{g}$. If we define the set of nonrandom nondecreasing functions as $\mathcal{I}(\reels^+,\reels^+)$ and the distribution of $\langle Y\rangle$ as $F_{\langle Y\rangle}$, we get $P_1^Y$ by integrating 
$$\proba (\sup_{0 \leq s \leq t} Y_s \geq g |\langle Y\rangle=y)$$ 
with respect to the value of $y \in \mathcal{I}(\reels^+,\reels^+)$. We obtain that (see Theorem \ref{theorem_rc1})
\begin{eqnarray*}
P_1^Y (t) & = & \int_{\mathcal{I}(\reels^+,\reels^+)} P_1^W \big(y_t\big) dF_{\langle Y\rangle}(y)\text{ for } t \geq 0.
\end{eqnarray*}

To apply the Dambis, Dubins-Schwarz theorem when the two-sided boundary is random and the quadratic variation is a stochastic process, we cannot rewrite the FPT to a random two-sided boundary as an equivalent FPT to a nonrandom two-sided boundary since there are two boundaries. However, we are able to adapt the arguments with a two-sided boundary. We define $u=(g,h,\langle Z\rangle)$, its distribution as $F_{u}$ and $\mathcal{S} = \mathcal{J} \times \mathcal{I}(\reels^+,\reels^+)$. We get $P_{g,h}^Z$ by integrating 
$$\proba (\sup_{0 \leq s \leq t} Z_s   \geq g \text{ or } \inf_{0 \leq s \leq t} Z_s  \leq h | u=(g_0,h_0,z))$$ with respect to the value of $(g_0,h_0,z) \in \mathcal{S}$.
If we assume that the stochastic process $Z$ is independent from the two-sided boundary $(g,h)$, we obtain that (see Theorem \ref{theorem_rc2})
\begin{eqnarray}
\label{eq_theorem_rc2}
P_{g,h}^Z (t) & = & \int_{\mathcal{S}} P_{g_0,h_0}^W \big(z_t\big) dF_{u}(g_0,h_0,z)\text{ for } t \geq 0.
\end{eqnarray}

We also derive an explicit solution for the inverse FPT problem. We consider the one-sided and two-sided boundary with (i) and (ii) cases. We derive the results in the general quadratic variation and absolutely continuous quadratic variation cases. The proofs are mainly based on topological argument in $\reels^+$ and the use of the explicit formula of boundary crossing probability (\ref{def_bcp})-(\ref{def_bcp2}).

We describe first the main results of the inverse FPT problem in the one-sided (i) case. If we define $F$ as the survival cdf, we obtain that the explicit solution is of the form (see Theorem \ref{theorem_F_nrc1})
\begin{eqnarray*}
\langle Z^F\rangle_{t} &=& (P_g^W)^{-1}(F(t)) \mathbf{1}_{\{0 < F(t) < 1\}} \text{ for } t \geq 0.
\end{eqnarray*} 
If we assume that the quadratic variation is absolutely continuous of the form $\langle Z^f\rangle_{t} = \int_0^t \sigma_{s,f}^2 ds \text{ for } t \geq 0$ with survival pdf $f$, we obtain that the explicit solution is of the form (see Theorem \ref{theorem_f_nrc1})
\begin{eqnarray*}
\label{def_expsol_f_nrc1}
\sigma_{t,f}^2 = &\frac{f(t)}{f_g^W((P_g^W)^{-1}(F(t)))} \mathbf{1}_{\{0 < F(t) < 1\}} & \text{ almost everywhere for } t \geq 0.
\end{eqnarray*}

In the two-sided (i) case, we obtain that the explicit solution is of the form (see Theorem \ref{theorem_F_nrc2})
\begin{eqnarray*}
\langle Z^F\rangle_{t} &=& (P_{g,h}^W)^{-1}(F(t)) \mathbf{1}_{\{0 < F(t) < 1\}} \text{ for } t \geq 0.
\end{eqnarray*} 
If we assume that the quadratic variation is absolutely continuous, we obtain that the explicit solution is of the form (see Theorem \ref{theorem_f_nrc2})
\begin{eqnarray*}
\sigma_{t,f}^2 = &\frac{f(t)}{f_{g,h}^W((P_{g,h}^W)^{-1}(F(t)))} \mathbf{1}_{\{0 < F(t) < 1\}} & \text{ almost everywhere for } t \geq 0.
\end{eqnarray*}

We describe now the main results of the inverse FPT problem in the one-sided (ii) case. If we define $F$ as the random survival cdf, we obtain that the explicit solution is of the form (see Theorem \ref{theorem_F_rc1})
\begin{eqnarray*}
\langle Y^F\rangle_{t} (\omega) &=& (P_1^W)^{-1}(F(t,\omega)) \mathbf{1}_{\{0 < F(t,\omega) < 1\}} \text{ for } t \geq 0 \text{ and } \omega \in \Omega.
\end{eqnarray*} 
If we assume that the random quadratic variation is absolutely continuous of the form $\langle Y^f\rangle_{t} (\omega) = \int_0^t \sigma_{s,f}^2(\omega) ds \text{ for } t \geq 0 \text{ and } \omega \in \Omega$ with random survival pdf $f$, we obtain that the explicit solution is of the form (see Theorem \ref{theorem_f_rc1})
\begin{eqnarray*}
\label{def_expsol_f_nrc1}
\sigma_{t,f}^2(\omega) = &\frac{f(t,\omega)}{f_1^W((P_1^W)^{-1}(F(t,\omega)))} \mathbf{1}_{\{0 < F(t,\omega) < 1\}} & \text{ almost everywhere for } t \geq 0 \text{ and } \omega \in \Omega.
\end{eqnarray*}

In the two-sided (ii) case, we obtain that the explicit solution is of the form (see Theorem \ref{theorem_F_rc2})
\begin{eqnarray*}
\langle Z^F\rangle_{t} (\omega)&=& (P_{g,h}^W)^{-1}(F(t,\omega)) \mathbf{1}_{\{0 < F(t,\omega) < 1\}} \text{ for } t \geq 0 \text{ and } \omega \in \Omega.
\end{eqnarray*} 
If we assume that the quadratic variation is absolutely continuous of the form $\langle Z^f\rangle_{t} (\omega) = \int_0^t \sigma_{s,f}^2(\omega) ds \text{ for } t \geq 0 \text{ and } \omega \in \Omega$, we obtain that the explicit solution is of the form (see Theorem \ref{theorem_f_rc2})
\begin{eqnarray*}
\sigma_{t,f}^2(\omega) = &\frac{f(t,\omega)}{f_{g,h}^W((P_{g,h}^W)^{-1}(F(t,\omega)))} \mathbf{1}_{\{0 < F(t,\omega) < 1\}} & \text{ almost everywhere for } t \geq 0 \text{ and } \omega \in \Omega.
\end{eqnarray*}

\subsection{Structure of this paper}
In what follows, we derive an explicit formula for the boundary crossing probability in Section \ref{section_expfor} and an explicit solution in Section \ref{section_application} in several cases. The proofs of the explicit formula are given in Section \ref{section_proofs_expfor}. The proofs of the inverse FPT problem can be found in Section \ref{section_proofs_application}. 
\section{Explicit formula}
\label{section_expfor}
In this section, we derive an explicit formula for the one-sided and two-sided boundary crossing probability (\ref{def_bcp})-(\ref{def_bcp2}) in the case (i) and (ii).
\subsection{One-sided (i) case}
In this part, we consider the case when the one-sided boundary is nonrandom constant, and the quadratic variation of the continuous local martingale is a time-varying function but not a stochastic process. 

We consider the complete stochastic basis $\mathcal{B} = (\Omega, \proba, \Sigma, \mathcal{F})$, where $\Sigma$ is a $\sigma$-field and $\mathcal{F} = (\mathcal{F}_t)_{t \in \reels^+}$ is a filtration. For $A \subset \reels^+$ and $B \subset \reels$ such that $0 \in A$, we define the set of nonrandom constant functions as $\mathcal{K}(A,B)$. We first give the definition of the set of boundary functions. 
\begin{definition}\label{defboundaryset_nrc1}
We define the set of nonrandom constant boundary functions as $\mathcal{G} = \mathcal{K} (\reels^+,\reels_*^+).$
\end{definition}
\noindent Since we consider constant boundary functions, we will abuse notation and identify $g \in \mathcal{G}$ with $g \in \reels_*^+$. We now give the definition of the FPT. 
\begin{definition}\label{defFPT_nrc1}
We define the FPT of an $\mathcal{F}$-adapted continuous stochastic process $(Z_t)_{t \in \reels^+}$ started at 0, i.e., with $Z_0=0$, to a nonrandom constant boundary
$g \in \mathcal{G}$ as
\begin{eqnarray}
\label{TgZdef_nrc1}
\Tau_g^Z = \inf \{t \in \reels^+ \text{ s.t. } Z_t \geq g\}.
\end{eqnarray}
\end{definition}
\noindent Since $Z$ is a continuous and 
$\mathcal{F}$-adapted stochastic process, and $\inf \{t \in \reels^+ \text{ s.t. } Z_t \geq g\} = \inf \{t \in \reels^+ \text{ s.t. } Z_t > g\} = \inf \{t \in \reels^+ \text{ s.t. } Z_t \in G\}$ where $G= \{(t,u) \in \reels^+ \times \reels \text{ s.t. } u > g\}$ is an open subset of $\reels^2$, the FPT $\Tau_{g}^{Z}$ is an $\mathcal{F}$-stopping time by Theorem I.1.28(a) (p. 7) in \cite{JacodLimit2003}. We can rewrite the boundary crossing probability $P_g^Z$ as the cumulative distribution function (cdf) of $\Tau_g^Z$, i.e., 
\begin{eqnarray}
\label{PgZdef_nrc1}
P_g^Z(t)= \proba (\Tau^Z_g \leq t) \text{ for any } t \geq 0.
\end{eqnarray}
If the cdf is absolutely continuous, we can also define its pdf $f_g^Z: \reels^+  \rightarrow  \reels^+$ as
\begin{eqnarray}
\label{fZgt_nrc1}
f_g^Z(t) & = & \frac{dP_g^Z(t)}{dt}.
\end{eqnarray}
We define an $\mathcal{F}$-adapted standard Wiener process as $(W_t)_{t \in \reels^+}$. We first consider the case when the stochastic process is a standard Wiener process, i.e., when $Z_t = W_t$ for any $t \in \reels^+$. The next lemma gives an explicit formula of $P_g^Z$ and $f_g^Z$, e.g., Levy distribution, which are known results by integrating the explicit formula conditioned on the final value of the Wiener process \cite{malmquist1954certain} (p. 526) with respect to the Wiener process final value as in \cite{wang1997boundary} (Equations (3), p. 55).
%We also define $\mathrm{ss}_{t}(v, w)$ as
%\begin{eqnarray}
%\label{sst}
% \displaystyle \mathrm{ss}_{t}(v, w)=\sum_{k=-\infty}^{\infty} \frac{w-v+2 k w}{\sqrt{2 \pi} t^{3 / 2}} e^{-(w-v+2 k w)^{2} / 2 t} \text{ for any } 0<v<w .  
%\end{eqnarray}
We define the standard normal cdf as 
\begin{eqnarray}
\label{def_standardgaussian}
\Phi (t)=\int_{-\infty}^{t} \frac{1}{\sqrt{2 \pi}} exp{\big(-\frac{u^2}{2}\big)}  du  \text{ for any } t \in \reels.
\end{eqnarray}
\begin{lemma}
\label{lemma_nrc1} We obtain a Levy distribution with $P_g^W (0)  =  0$, $f_g^W (0)  =  0$, 
\begin{eqnarray}
 \label{eqP_lemma_nrc1}
P_g^W (t) & = & 1 - \Phi \left(\frac{g}{\sqrt{t}}\right) + \Phi \left(\frac{-g}{\sqrt{t}}\right) \text{ for } t>0,\\
\label{eqf_lemma_nrc1} f_g^W (t) & = & \frac{g}{\sqrt{2 \pi t^3}} e^{\frac{-g^2}{2t}} \text{ for } t>0.
\end{eqnarray}
%In the upper and lower boundary case, we obtain that $f_g^W (0)  =  0$, $P_g^W (0)  =  0$,
%\begin{eqnarray}
%\label{fgWex21} f_g^W (t) & = & \operatorname{ss}_t\left(g,g-g^{(1)}\right)+\operatorname{ss}_t\left(-g^{(1)},g-g^{(1)}\right),\\
%P_g^W (t) & = &\sum_{k=-\infty}^{\infty}{\left(2-\operatorname{erf}\left(\frac{-g^{(1)}+2k(g-g^{(1)})}{\sqrt{2t}}\right)-\operatorname{erf}\left(\frac{g+2k(g-g^{(1)})}{\sqrt{2t}}\right)\right)}.
%\end{eqnarray}
\end{lemma}
\noindent The main result of this paper in the one-sided (i) case, i.e., Theorem \ref{theorem_nrc1}, states that the boundary crossing probability of a continuous local martingale to a constant boundary is equal to the boundary crossing probability of a standard Wiener process to a constant boundary up to a time change of quadratic variation value. It relies on the constancy of the boundary and Dambis, Dubins-Schwarz theorem for continuous local martingale (Th. V.1.6 in \cite{revuz2013continuous}).  Accordingly, we provide the assumption on the continuous local martingale which is required to apply Dambis, Dubins-Schwarz theorem.
\begin{assumption}
\label{assDDS_nrc1} We assume that $Z$ is a continuous $\mathcal{F}$-adapted local martingale with nonrandom quadratic variation $\langle Z\rangle$ and such that $Z_0=0$ and $\langle Z\rangle_{\infty} = \infty$.
\end{assumption}
\noindent For any function $h:  \reels^+ \rightarrow  \reels^+$, $a \mapsto h(a)$, and any $A \subset \reels^+$, we define the restriction of $h$ to $A$ as $h \restriction_A$ such that $h \restriction_A:  A \rightarrow  \reels^+$, $ a \mapsto h(a)$. For  $p \in \reels$, $p \geq 1$, we define the set of p-integrable non random functions and the set of locally p-integrable non random functions as respectively
\begin{eqnarray*}
L_{p}(A) & =& \bigl\{h:A\to\reels^+ \text{ measurable s.t. }\,\int_A | h(x)|^p \,\mathrm{d}x <+\infty\bigr\},\\
L_{p,\mathrm{loc}}(A) & =& \bigl\{h:A\to\reels^+ \text{ measurable s.t. }\,h \restriction_K \in L_p(K)\ \forall\, K \subset A,\, K \text{ compact}\bigr\},
\end{eqnarray*}
$\text{ for } A \subset \reels^+$ measurable.
\begin{example}
\label{example_ito_nrc1}
We can consider a continuous It\^{o} process with no drift of the form 
\begin{eqnarray}
\label{Xdef}
Z_t = \int_0^t \sigma_s dW_s \text{ for } t \geq 0,
\end{eqnarray}
where $\sigma : \reels^+ \rightarrow \reels $ is a nonrandom function. If we assume that $\sigma \in L_{2,\mathrm{loc}}(\reels^{+})$, then $Z$ is a local martingale with nonrandom quadratic variation $\langle Z\rangle_{t} = \int_0^t \sigma_u^2 du$ by Theorem I.4.40 (p. 48) in \cite{JacodLimit2003}. If we further assume that $\int_0^t \sigma_u^2 du \rightarrow \infty$ as $t \rightarrow \infty$, we have that $Z$ satisfies Assumption \ref{assDDS_nrc1}.
\end{example}
\noindent We introduce Proposition \ref{lemmaDDS_nrc1} in what follows. The main idea to prove it is the scale invariant property of the time-changed Wiener process and thus the scale invariant property of the FPT. More specifically, we can apply by Assumption \ref{assDDS_nrc1} the Dambis, Dubins-Schwarz theorem for continuous local martingale. We define the generalized inverse function of the non-decreasing function $\langle Z\rangle$ for any $t \geq0$ as 
$$\langle Z\rangle_{t}^{-1}=\inf\{s \geq 0 \text{ s.t. }\langle Z\rangle_{s}>t\}.$$ 
We also define the canonical filtration of any stochastic process $Z$ as $\mathcal{F}_t^Z = \sigma \big(Z(C), C \in \mathcal{B}(\reels^+), C \subset [0,t] \big)$, where $\mathcal{B}(\reels^+)$ is the usual Borel $\sigma$-field generated by the open sets of $\reels^+$. Then $(B_{t})_{t\geq 0}=(Z_{\langle Z\rangle_{t}^{-1}})_{t\geq 0}$ is a $({\mathcal {F}}_t^{\langle Z\rangle^{-1}})_{t\geq 0}$-Wiener process and 
\begin{eqnarray}
\label{proof230110}
Z_{t}=B_{\langle Z\rangle_{t}} \text{ for } t \geq 0.
\end{eqnarray}
\begin{proposition}
\label{lemmaDDS_nrc1}
Under Assumption \ref{assDDS_nrc1}, we have that
\begin{eqnarray}
\label{eq_lemmaDDS_nrc1}
\big\{ \Tau_g^Z =t \big\} & = & \big\{ \Tau_g^B = \langle Z\rangle_{t} \big\} \text{ for } t \geq 0.
\end{eqnarray}
\end{proposition}
\noindent In what follows, we state the main result of this paper in the one-sided (i) case. The proof is mainly based on Proposition \ref{lemmaDDS_nrc1}.
\begin{theorem}
\label{theorem_nrc1}
Under Assumption \ref{assDDS_nrc1}, we have that
\begin{eqnarray}
\label{eq_theorem_nrc1}
P_g^Z (t) & = & P_g^W \big(\langle Z\rangle_{t}\big) \text{ for } t \geq 0.
\end{eqnarray}
\end{theorem}
\noindent As a corollary, we obtain the pdf from the FPT of a continuous local martingale to a constant boundary if we assume that the quadratic variation is absolutely continuous.
\begin{corollary}
\label{corollary_nrc1}
Under Assumption \ref{assDDS_nrc1} and if we assume that the quadratic variation $\langle Z\rangle$ is absolutely continuous on $\reels^+$, we have that
\begin{eqnarray}
\label{eq_corollary_nrc1}
f_g^X (t) & = & \langle Z\rangle_t ' f_g^W \big(\langle Z\rangle_t\big) \text{ for } t \geq 0.
\end{eqnarray}
\end{corollary}

\subsection{Two-sided (i) case}
In this part, we consider the case when the two-sided boundary is nonrandom constant, and the quadratic variation of the continuous local martingale is a time-varying function but not a stochastic process. 

We first give the definition of the set of two-sided boundary functions. 
\begin{definition}\label{defboundaryset_nrc2}
We define the set of nonrandom constant two-sided boundary functions as $\mathcal{H} = \mathcal{K} (\reels^+,\reels_*^+) \times \mathcal{K} (\reels^+,\reels_*^-).$
\end{definition} 
\begin{definition}\label{defFPT_nrc2}
We define the FPT of an $\mathcal{F}$-adapted continuous stochastic process $(Z_t)_{t \in \reels^+}$ started at 0, i.e., with $Z_0=0$, to a nonrandom constant two-sided boundary
$(g,h) \in \mathcal{H}$ as
\begin{eqnarray}
\label{TgZdef_nrc2}
\Tau_{g,h}^Z = \inf \{t \in \reels^+ \text{ s.t. } Z_t \geq g \text{ or } Z_t \leq h \}.
\end{eqnarray}
\end{definition}
\noindent Since $Z$ is a continuous and 
$\mathcal{F}$-adapted stochastic process and $\inf \{t \in \reels^+ \text{ s.t. } Z_t \geq g \text{ or } Z_t \leq h \} = \inf \{t \in \reels^+ \text{ s.t. } Z_t > g \text{ or } Z_t < h \} = \inf \{t \in \reels^+ \text{ s.t. } Z_t \in G\}$ where $G= \{(t,u) \in \reels^+ \times \reels \text{ s.t. } u > g \text{ or } u < h \}$ is an open subset of $\reels^2$, the FPT $\Tau_{g,h}^{Z}$ is an $\mathcal{F}$-stopping time by Theorem I.1.28(a) (p. 7) in \cite{JacodLimit2003}. We can rewrite the boundary crossing probability $P_{g,h}^Z$ as the cdf of $\Tau_{g,h}^Z$, i.e., 
\begin{eqnarray}
\label{PgZdef_nrc2}
P_{g,h}^Z(t)= \proba (\Tau^Z_{g,h} \leq t) \text{ for any } t \geq 0.
\end{eqnarray}
If the cdf is absolutely continuous, we can also define its pdf $f_{g,h}^Z: \reels^+  \rightarrow  \reels^+$ as
\begin{eqnarray}
\label{fZgt_nrc2}
f_{g,h}^Z(t) & = & \frac{dP_{g,h}^Z(t)}{dt}.
\end{eqnarray}
We first consider the case when the stochastic process is a standard Wiener process, i.e., when $Z_t = W_t$ for any $t \in \reels^+$. The next lemma gives an explicit formula of $P_{g,h}^Z$ and $f_{g,h}^Z$ which are respectively known results from Theorem 4.3 (p. 180) and Theorem 5.1 (p. 191) in \cite{anderson1960modification}.
We define $\mathrm{ss}_{t}(v, w)$ as
\begin{eqnarray}
\label{sst}
 \displaystyle \mathrm{ss}_{t}(v, w)=\sum_{k=-\infty}^{\infty} \frac{w-v+2 k w}{\sqrt{2 \pi} t^{3 / 2}} e^{-(w-v+2 k w)^{2} / 2 t} \text{ for any } 0<v<w .  
\end{eqnarray}

\begin{lemma}
\label{lemma_nrc2} We obtain that $P_{g,h}^W (0)  =  0$, $f_{g,h}^W (0)  =  0$, 
\begin{eqnarray}
\label{eqP_lemma_nrc2} P_{g,h}^W (t) & = &\sum_{k=-\infty}^{\infty}{\left(4-2\phi\left(\frac{-h+2k(g-h)}{\sqrt{t}}\right)-2\phi\left(\frac{g+2k(g-h)}{\sqrt{t}}\right)\right)},\\
\label{eqf_lemma_nrc2} f_{g,h}^W (t) & = & \operatorname{ss}_t\left(g,g-h\right)+\operatorname{ss}_t\left(-h,g-h\right).
\end{eqnarray}
\end{lemma}

\noindent The main result of this paper adapts to the two-sided (i) case since the arguments of boundary constancy and Dambis, Dubins-Schwarz theorem still hold. More specifically, we introduce Proposition \ref{lemmaDDS_nrc1} in what follows. We can apply by Assumption \ref{assDDS_nrc1} the Dambis, Dubins-Schwarz theorem for continuous local martingale. We define the generalized inverse function of the non-decreasing function $\langle Z\rangle$ for any $t \geq0$ as 
$$\langle Z\rangle_{t}^{-1}=\inf\{s \geq 0 \text{ s.t. }\langle Z\rangle_{s}>t\}.$$ 
Then $(B_{t})_{t\geq 0}=(Z_{\langle Z\rangle_{t}^{-1}})_{t\geq 0}$ is a $({\mathcal {F}}_t^{\langle Z\rangle^{-1}})_{t\geq 0}$-Wiener process and 
\begin{eqnarray}
\label{proof230110_nrc2}
Z_{t}=B_{\langle Z\rangle_{t}} \text{ for } t \geq 0.
\end{eqnarray}
\begin{proposition}
\label{lemmaDDS_nrc2}
Under Assumption \ref{assDDS_nrc1}, we have that
\begin{eqnarray}
\label{eq_lemmaDDS_nrc2}
\big\{ \Tau_{g,h}^Z =t \big\} & = & \big\{ \Tau_{g,h}^B = \langle Z\rangle_{t} \big\} \text{ for } t \geq 0.
\end{eqnarray}
\end{proposition}
\noindent In what follows, we state the main result of this paper in the two-sided (i) case.
\begin{theorem}
\label{theorem_nrc2}
Under Assumption \ref{assDDS_nrc1}, we have that
\begin{eqnarray}
\label{eq_theorem_nrc2}
P_{g,h}^Z (t) & = & P_{g,h}^W \big(\langle Z\rangle_{t}\big) \text{ for } t \geq 0.
\end{eqnarray}
\end{theorem}
\noindent As a corollary, we obtain the pdf from the FPT of a continuous local martingale to a constant boundary if we assume that the quadratic variation is absolutely continuous.
\begin{corollary}
\label{corollary_nrc2}
Under Assumption \ref{assDDS_nrc1} and if we assume that the quadratic variation $\langle Z\rangle$ is absolutely continuous on $\reels^+$, we have that
\begin{eqnarray}
\label{eq_corollary_nrc2}
f_{g,h}^X (t) & = & \langle Z\rangle_t ' f_{g,h}^W \big(\langle Z\rangle_t\big) \text{ for } t \geq 0.
\end{eqnarray}
\end{corollary}

\subsection{One-sided (ii) case}
In this part, we consider the case when the one-sided boundary is random constant, and the quadratic variation of the continuous local martingale is a stochastic process.

\begin{definition}\label{defboundaryset_rc1}
We define the set of random constant boundary functions as $\mathcal{I} = \big\{ \reels^+ \times \Omega \rightarrow \reels_*^+$ such that for any $g \in \mathcal{I}$ and $\omega \in \Omega$ we have $g(\omega) \in \mathcal{G}$\big\}.
\end{definition}
 
\begin{definition}\label{defFPT_rc1}
We define the FPT of an $\mathcal{F}$-adapted continuous stochastic process $(Z_t)_{t \in \reels^+}$ started at 0, i.e., with $Z_0=0$, to a random constant boundary
$g \in \mathcal{I}$ as
\begin{eqnarray}
\label{TgZdef_rc1}
\Tau_g^Z = \inf \{t \in \reels^+ \text{ s.t. } Z_t \geq g\}.
\end{eqnarray}
\end{definition}
\noindent Since $Z-g$ is an $\mathcal{F}$-adapted continuous stochastic process and $\inf \{t \in \reels^+ \text{ s.t. } Z_t \geq g\} = \inf \{t \in \reels^+ \text{ s.t. } Z_t - g \geq 0\} = \inf \{t \in \reels^+ \text{ s.t. } Z_t - g > 0\} = \inf \{t \in \reels^+ \text{ s.t. } Z_t - g \in \reels^+_*\}$, the FPT $\Tau_{g}^{Z}$ is a $\mathcal{F}$-stopping time by Theorem I.1.28(a) (p. 7) in \cite{JacodLimit2003}. We can rewrite the boundary crossing probability $P_g^Z$ as the cdf of $\Tau_g^Z$, i.e., 
\begin{eqnarray}
\label{PgZdef_rc1}
P_g^Z(t)= \proba (\Tau^Z_g \leq t) \text{ for any } t \geq 0.
\end{eqnarray}
If the cdf is absolutely continuous, we can also define its pdf $f_g^Z: \reels^+  \rightarrow  \reels^+$ as
\begin{eqnarray}
\label{fZgt_rc1}
f_g^Z(t) & = & \frac{dP_g^Z(t)}{dt}.
\end{eqnarray}
\noindent To apply the Dambis, Dubins-Schwarz theorem in the one-sided (ii) case, the main idea is to rewrite the FPT to a random boundary as an equivalent FPT to a nonrandom boundary. More specifically, if we define the new process as $Y = \frac{Z}{g}$, we observe that the FPT (\ref{TgZdef_rc1}) may be rewritten as 
\begin{eqnarray}
\Tau_g^Z = \Tau_{1}^Y.
\end{eqnarray}

\begin{assumption}
\label{assDDS_rc1}
We assume that $Y$ is a continuous $\mathcal{F}$-local martingale with random quadratic variation $\langle Y\rangle$ and such that $Y_0=0$ and $\langle Y\rangle_{\infty} = \infty$.
\end{assumption} 
\noindent For any stochastic process $h:  \reels^+ \times \Omega \rightarrow  \reels^+$, $a \mapsto h(a)$, and any $A \subset \reels^+ \times \Omega$, we define the restriction of $h$ to $A$ as $h \restriction_A$ such that $h \restriction_A:  A \rightarrow  \reels^+$, $ a \mapsto h(a)$. For  $p \in \reels$, $p \geq 1$, we define the set of p-integrable stochastic processes and the set of locally p-integrable stochastic processes as respectively
\begin{eqnarray*}
L_{p}(A) & =& \bigl\{h:A\to\reels^+ \text{ measurable s.t. }\,\int_A | h(x)|^p \,\mathrm{d}x <+\infty\bigr\},\\
L_{p,\mathrm{loc}}(A) & =& \bigl\{h:A\to\reels^+ \text{ measurable s.t. }\,h \restriction_K \in L_p(K)\ \forall\, K \subset A,\, K \text{ compact}\bigr\},
\end{eqnarray*}
$\text{ for } A \subset \reels^+ \times \Omega$ measurable.
\begin{example}
\label{example_ito_rc1}
We can consider a continuous It\^{o} process with no drift of the form 
\begin{eqnarray}
\label{eq_example_rc1}
Y_t = \int_0^t \sigma_s dW_s \text{ for } t \geq 0,
\end{eqnarray}
where $\sigma : \reels^+ \times \Omega \rightarrow \reels $ is a predictable process on $\mathcal{B}$ such that the integral defined in Equation (\ref{eq_example_rc1}) is well-defined. If we assume that $\sigma \in L_{2,\mathrm{loc}}(\reels^{+} \times \Omega)$, then $Y$ is a local martingale with random quadratic variation $\langle Y\rangle_{t} = \int_0^t \sigma_u^2 du$ by Theorem I.4.40 (p. 48) in \cite{JacodLimit2003}. If we further assume that $\int_0^t \sigma_u^2 du \rightarrow \infty$ as $t \rightarrow \infty$, we have that $Y$ satisfies Assumption \ref{assDDS_rc1}.
\end{example}

\noindent We introduce Proposition \ref{lemmaDDS_rc1} in what follows. We can apply by Assumption \ref{assDDS_rc1} the Dambis, Dubins-Schwarz theorem for continuous local martingale. We define the generalized inverse stochastic process of the non-decreasing stochastic process $\langle Y\rangle$ for any $t \geq0$ as 
$$\langle Y\rangle_{t}^{-1}=\inf\{s \geq 0 \text{ s.t. }\langle Y\rangle_{s}>t\}.$$ 
Then $(B_{t})_{t\geq 0}=(Y_{\langle Y\rangle_{t}^{-1}})_{t\geq 0}$ is a $({\mathcal {F}}_t^{\langle Y\rangle^{-1}})_{t\geq 0}$-Wiener process and 
\begin{eqnarray}
\label{proof230110_rc1}
Y_{t}=B_{\langle Y\rangle_{t}}.
\end{eqnarray}
\begin{proposition}
\label{lemmaDDS_rc1}
Under Assumption \ref{assDDS_rc1}, we have that
\begin{eqnarray}
\label{eq_lemmaDDS_rc1}
\big\{ \Tau_1^Y =t \big\} & = & \big\{ \Tau_1^B = \langle Y\rangle_{t} \big\} \text{ for } t \geq 0.
\end{eqnarray}
\end{proposition}
\noindent In what follows, we state the main result of this paper in the one-sided (ii) case. For $A \subset \reels^+$ and $B \subset \reels$ such that $0 \in A$, we define the set of nonrandom nondecreasing functions as $\mathcal{I}(A,B)$. When seen as a function of $\omega$, the arrival space of $\langle Y\rangle$ is $\mathcal{I}(\reels^+,\reels^+)$. We define the distribution of $\langle Y\rangle$ as $F_{\langle Y\rangle}$. We get $P_1^Y$ in the next theorem by integrating $\proba (\Tau^{Y}_1 \leq t |\langle Y\rangle=y)$ with respect to the value of $y \in \mathcal{I}(\reels^+,\reels^+)$.
\begin{theorem}
\label{theorem_rc1}
Under Assumption \ref{assDDS_rc1}, we have that
\begin{eqnarray}
\label{eq_theorem_rc1}
P_1^Y (t) & = & \int_{\mathcal{I}(\reels^+,\reels^+)} P_1^W \big(y_t\big) dF_{\langle Y\rangle}(y)\text{ for } t \geq 0.
\end{eqnarray}
\end{theorem}
\noindent As a corollary, we obtain the pdf from the FPT of a continuous local martingale to a constant boundary if we assume that the quadratic variation is absolutely continuous.
\begin{corollary}
\label{corollary_rc1}
Under Assumption \ref{assDDS_rc1} and if we assume that the quadratic variation $\langle Y\rangle$ is absolutely continuous on $\reels^+$, we have that
\begin{eqnarray}
\label{eq_corollary_rc1}
f_1^Y (t) & = & \int_{\mathcal{I}(\reels^+,\reels^+)} y_t ' f_1^W \big(y_t\big) dF_{\langle Y\rangle}(y) \text{ for } t \geq 0.
\end{eqnarray}
\end{corollary}

\subsection{Two-sided (ii) case}
In this part, we consider the case when the two-sided boundary is random constant, and the quadratic variation of the continuous local martingale is a stochastic process.
\begin{definition}\label{defboundaryset_rc2}
We define the set of random constant two-sided boundary functions as $\mathcal{J} = \big\{ \reels^+ \times \Omega \rightarrow \reels_*^+ \times \reels_*^-$ such that for any $(g,h) \in \mathcal{J}$ and $\omega \in \Omega$ we have $g(\omega) \in \mathcal{G}$ and $-h(\omega) \in \mathcal{G}$\big\}.
\end{definition}
 
\begin{definition}\label{defFPT_rc2}
We define the FPT of an $\mathcal{F}$-adapted continuous stochastic process $(Z_t)_{t \in \reels^+}$ started at 0, i.e., with $Z_0=0$, to a random constant two-sided boundary
$(g,h) \in \mathcal{J}$ as
\begin{eqnarray}
\label{TgZdef_rc2}
\Tau_{g,h}^Z = \inf \{t \in \reels^+ \text{ s.t. } Z_t \geq g \text{ or } Z_t \leq h \}.
\end{eqnarray}
\end{definition}
\noindent We can rewrite $\Tau_{g,h}^Z$ as the infimum of two $\mathcal{F}$-stopping times, i.e., $\Tau_{g,h}^Z = \inf (\Tau_{g}^Z,\Tau_{-h}^{-Z})$ thus it is an $\mathcal{F}$-stopping time. We can rewrite the boundary crossing probability $P_{g,h}^Z$ as the cdf of $\Tau_{g,h}^Z$, i.e., 
\begin{eqnarray}
\label{PgZdef_rc2}
P_{g,h}^Z(t)= \proba (\Tau^Z_{g,h} \leq t) \text{ for any } t \geq 0.
\end{eqnarray}
If the cdf is absolutely continuous, we can also define its pdf $f_{g,h}^Z: \reels^+  \rightarrow  \reels^+$ as
\begin{eqnarray}
\label{fZgt_rc2}
f_{g,h}^Z(t) & = & \frac{dP_{g,h}^Z(t)}{dt}.
\end{eqnarray}
\noindent To apply the Dambis, Dubins-Schwarz theorem in the two-bounded (ii) case, we cannot rewrite the FPT to a random two-sided boundary as an equivalent FPT to a nonrandom two-sided boundary since there are two boundaries. However, we are able to adapt the arguments with a two-sided boundary. 
\begin{assumption}
\label{assDDS_rc2}
We assume that $Z$ is a continuous $\mathcal{F}$-local martingale with random quadratic variation $\langle Z\rangle$ and such that $Z_0=0$ and $\langle Z\rangle_{\infty} = \infty$.
\phantom\qedhere
\end{assumption} 
\noindent We introduce Proposition \ref{lemmaDDS_rc2} in what follows. We can apply by Assumption \ref{assDDS_rc2} the Dambis, Dubins-Schwarz theorem for continuous local martingale. We define the generalized inverse stochastic process of the non-decreasing stochastic process $\langle Z\rangle$ for any $t \geq0$ as 
$$\langle Z\rangle_{t}^{-1}=\inf\{s \geq 0 \text{ s.t. }\langle Z\rangle_{s}>t\}.$$ 
Then $(B_{t})_{t\geq 0}=(Z_{\langle Z\rangle_{t}^{-1}})_{t\geq 0}$ is a $({\mathcal {F}}_t^{\langle Z\rangle^{-1}})_{t\geq 0}$-Wiener process and 
\begin{eqnarray}
\label{proof230110_rc2}
Z_{t}=B_{\langle Z\rangle_{t}} \text{ for } t \geq 0.
\end{eqnarray}
\begin{proposition}
\label{lemmaDDS_rc2}
Under Assumption \ref{assDDS_rc2}, we have that
\begin{eqnarray}
\label{eq_lemmaDDS_rc2}
\big\{ \Tau_{g,h}^Z =t \big\} & = & \big\{ \Tau_{g,h}^B = \langle Z\rangle_{t} \big\} \text{ for } t \geq 0.
\end{eqnarray}
\end{proposition}
\noindent In what follows, we state the main result of this paper in the two-sided (ii) case. We define $u=(g,h,\langle Z\rangle)$, its distribution as $F_{u}$ and $\mathcal{S} = \mathcal{J} \times \mathcal{I}(\reels^+,\reels^+)$. We get $P_{g,h}^Z$ in the next theorem by integrating $\proba (\Tau^{Z}_{g,h} \leq t | u=(g_0,h_0,z))$ with respect to the value of $(g_0,h_0,z) \in \mathcal{S}$.
\begin{theorem}
\label{theorem_rc2}
Under Assumption \ref{assDDS_rc2} and if we assume that the stochastic process $Z$ is independent from the two-sided boundary $(g,h)$, we have that
\begin{eqnarray}
\label{eq_theorem_rc2}
P_{g,h}^Z (t) & = & \int_{\mathcal{S}} P_{g_0,h_0}^W \big(z_t\big) dF_{u}(g_0,h_0,z)\text{ for } t \geq 0.
\end{eqnarray}
\end{theorem}
\noindent As a corollary, we obtain the pdf from the FPT of a continuous local martingale to a constant boundary if we assume that the quadratic variation is absolutely continuous.
\begin{corollary}
\label{corollary_rc2}
Under Assumption \ref{assDDS_rc2} and if we assume that the stochastic process $Z$ is independent from the two-sided boundary $(g,h)$ and that the quadratic variation $\langle Z\rangle$ is absolutely continuous on $\reels^+$, we have that
\begin{eqnarray}
\label{eq_corollary_rc2}
f_{g,h}^Z (t) & = & \int_{\mathcal{S}} z_t ' f_{g_0,h_0}^W \big(z_t\big) dF_{u}(g_0,h_0,z) \text{ for } t \geq 0.
\end{eqnarray}
\end{corollary}

\section{Applications in the inverse FPT problem}
\label{section_application}
In this section, we derive an explicit solution of the inverse FPT problem for the one-sided and two-sided boundary and in the case (i) and (ii).
\subsection{One-sided (i) case} 
We define the error function and its inverse as
\begin{eqnarray}
\label{def_erf}
\operatorname{erf}(t)&=&\frac{2}{\sqrt{\pi}}\int_0^t{e^{-u^2}du} \text{ for any } t \in \reels,\\
\label{erfinv} \operatorname{erf}(\operatorname{erfinv}(t))&=&z \text{ for any } t \in (-1,1).
\end{eqnarray}
The first lemma shows that there exists an invert of $P_g^W$ which we denote $(P_g^W)^{-1}$ and gives explicit formulae of $(P_g^W)^{-1}(t)$ and $f_g^W((P_g^{W})^{-1}(t))$ for $0 \leq t < 1$, all of which are new results which will be useful to express the explicit solution of the inverse problem. The proof relies on Lemma \ref{lemma_nrc1}.
\begin{lemma}
\label{lemma_inv_nrc1} There exists an invert of $P_g^W$ which we denote $(P_g^W)^{-1} : [0,1) \rightarrow \reels^+$ and is strictly increasing such that 
$(P_g^{W})^{-1}(0)=0$,
\begin{eqnarray}
\label{PgWinvt}
(P_g^{W})^{-1}(t)&=&\frac{g^2}{2\operatorname{erfinv}(1-t)^2} \text{ for } 0<t<1.
\end{eqnarray}
Finally, we have $f_g^W((P_g^{W})^{-1}(t))=0$,
\begin{equation}
\label{fP1}
f_g^W((P_g^{W})^{-1}(t))=\frac{2}{g^2\sqrt{\pi}}{\operatorname{erfinv}(1-t)}^3e^{-{\operatorname{erfinv}(1-t)}^2} \text{ for } 0<t<1.
\end{equation}
%In the upper and lower boundary case, we obtain that $f_g^W (0)  =  0$, $P_g^W (0)  =  0$,
%\begin{eqnarray}
%\label{fgWex21} f_g^W (t) & = & \operatorname{ss}_t\left(g,g-g^{(1)}\right)+\operatorname{ss}_t\left(-g^{(1)},g-g^{(1)}\right),\\
%P_g^W (t) & = &\sum_{k=-\infty}^{\infty}{\left(2-\operatorname{erf}\left(\frac{-g^{(1)}+2k(g-g^{(1)})}{\sqrt{2t}}\right)-\operatorname{erf}\left(\frac{g+2k(g-g^{(1)})}{\sqrt{2t}}\right)\right)}.
%\end{eqnarray}
\end{lemma}
\noindent To define the inverse FPT problem, we introduce the set of survival cdfs. Since the stochastic process $Z$ is continuous and thus its quadratic variation $\langle Z\rangle$ is also continuous, we accordingly consider the set of continuous cdfs.
\begin{definition} \label{Fdef_nrc1}
A function $F:\reels^+ \rightarrow  [0,1]$ is a survival cdf if $F$ is nondecreasing, continuous, 
and satisfies $F(0) = 0$ and $\underset{t \rightarrow \infty}{\lim} F(t) =1$. 
\end{definition}
\noindent By Equation (\ref{eq_theorem_nrc1}) from Theorem \ref{theorem_nrc1}, we have that $P_g^Z (t)  = P_g^W \big(\langle Z\rangle_{t}\big) \text{ for } t \geq 0$.
We first give the definition of a solution in the inverse first-passage time problem of quadratic variation.
\begin{definition}
\label{def_solution_F_nrc1}
For any cdf $F$, we say that a nonrandom nondecreasing function $
\tilde{v}_{F} : \reels^+ \rightarrow \reels^+ $ which is the quadratic variation of a continuous local martingale $Z^F$, i.e.,
\begin{eqnarray}
\label{tildev_F_def_nrc1}
\langle Z^F\rangle_{t} &=& \tilde{v}_{F}(t) \text{ for } t \geq 0,
\end{eqnarray}
is solution if it satisfies
\begin{eqnarray}
\label{PgZF_nrc1}
P_g^{Z^F} (t) &=& F(t) \text{ for } t \geq 0.
\end{eqnarray}
\end{definition}
\noindent Equation (\ref{tildev_F_def_nrc1}) in Definition \ref{def_solution_F_nrc1} implicitly implies the existence of a continuous local martingale with quadratic variation $\tilde{v}_{F}$, but this is true since any standard Wiener process with time-changed $\frac{\tilde{v}_{F}(t)}{t}$ will have $\tilde{v}_{F}$ as quadratic variation. We now give the definition of the explicit solution.
\begin{definition}
\label{def_explicitsolution_F_nrc1}
For any cdf $F$, we say that a nonrandom nondecreasing function $v_F$ which is the quadratic variation of a continuous local martingale $Z^F$, i.e.,
\begin{eqnarray}
\label{v_F_def_nrc1}
\langle Z^F\rangle_{t} &=& v_{F}(t) \text{ for } t \geq 0,
\end{eqnarray} is an explicit solution if it is of the form
\begin{eqnarray}
\label{def_expsol_F_nrc1}
v_{F}(t) = &(P_g^W)^{-1}(F(t)) \mathbf{1}_{\{0 < F(t) < 1\}} & \text{ for } t \geq 0.
\end{eqnarray}
\end{definition}
\noindent If we substitute $(P_g^{W})^{-1}$ in Equation (\ref{def_expsol_F_nrc1}) with Equation (\ref{PgWinvt}) from Lemma \ref{lemma_inv_nrc1}, we can reexpress the explicit solution as 
\begin{eqnarray}
\label{expsoldefsimp_nrc1}
v_{F}(t) = & \frac{g^2}{2\operatorname{erfinv}(1-F(t))^2} \mathbf{1}_{\{0 < F(t) < 1\}} & \text{ for } t \geq 0.
\end{eqnarray}
We define the infimum value such that $F(t)$ is positive and the infimum value such that $F(t)$ equals unity as
  \begin{eqnarray}
 \label{defKF0}
 K_F^0 & = & \inf \{ t > 0 \text{ such that } F(t)>0 \} \text{ and }\\
 \label{defKF1}
 K_F^1 & = & \inf \{ t > 0 \text{ such that } F(t)=1 \}.
\end{eqnarray}
Let us give an assumption sufficient for $Z^{F}$ to satisfy Assumption \ref{assDDS_nrc1}.
\begin{assumption}
\label{assmain_F_nrc1} We assume that $K_F^1$ is not finite.
\end{assumption}
\noindent The following theorem states that under Assumption \ref{assmain_F_nrc1}, (a) $Z^F$ satisfies Assumption \ref{assDDS_nrc1} and (b) that any nondecreasing function is solution if and only if it is an explicit solution. The proof of (b) is based on substituting the left-hand side of Equation (\ref{PgZF_nrc1}) with Equations (\ref{eq_theorem_nrc1}) and (\ref{v_F_def_nrc1}) and then inverting on both sides of the equation to derive the explicit solution.
\begin{theorem}
\label{theorem_F_nrc1}
Under Assumption \ref{assmain_F_nrc1}, (a) $Z^{F}$ satisfies Assumption \ref{assDDS_nrc1} (b) (i) $v_F$ is a solution in the sense of Definition \ref{def_solution_F_nrc1} $\iff$ (ii) $v_F$ is an explicit solution in the sense of Definition \ref{def_explicitsolution_F_nrc1}.
\end{theorem}
\noindent In what follows, we consider the particular case when the quadratic variation $\langle Z\rangle$ is absolutely continuous.  We accordingly consider the set of absolutely continuous cdfs.
\begin{definition} \label{fdef_nrc1}
A function $f:\reels^+ \rightarrow \reels^+$ is a survival pdf if it satifies
\begin{eqnarray}
\label{Ffrelation_nrc1}
F(t) & = &\int_0^t f(s) ds \text{ for } t \geq 0.
\end{eqnarray}
\end{definition}
\noindent We give the definition of a solution in the inverse first-passage time problem of absolutely continuous quadratic variation. Since the quadratic variation is absolutely continuous, we can consider its derivative in the solution definition.
\begin{definition}
\label{def_solution_f_nrc1}
For any pdf $f$, we say that a variance function $\widetilde{\sigma}_{f}^2 : \reels^+ \rightarrow \reels^+ $ which is the quadratic variation  derivative of a continuous local martingale $Z^f$, i.e.,
\begin{eqnarray}
\label{tildesigma_f_def_nrc1}
\langle Z^f\rangle_{t} &=& \int_0^t \widetilde{\sigma}_{s,f}^2 ds \text{ for } t \geq 0,
\end{eqnarray}
is solution if it satisfies
\begin{eqnarray}
\label{PgZf_nrc1}
P_g^{Z^f} (t) &=& F(t) \text{ for } t \geq 0.
\end{eqnarray}
\end{definition}
\noindent Equation (\ref{tildesigma_f_def_nrc1}) in Definition \ref{def_solution_f_nrc1} implicitly implies the existence of a continuous local martingale with quadratic variation $\int_0^t \widetilde{\sigma}_{s,f}^2 ds$, but this is true since we can consider It\^{o} processes from Example \ref{example_ito_nrc1}. We now give the definition of the explicit solution.
\begin{definition}
\label{def_explicitsolution_f_nrc1}
For any pdf $f$, we say that a variance function $\sigma_{f}^2 : \reels^+ \rightarrow \reels^+ $ which is the quadratic variation  derivative of a continuous local martingale $Z^f$, i.e.,
\begin{eqnarray}
\label{sigma_f_def_nrc1}
\langle Z^f\rangle_{t} &=& \int_0^t \sigma_{s,f}^2 ds \text{ for } t \geq 0,
\end{eqnarray} 
is an explicit solution if it is of the form
\begin{eqnarray}
\label{def_expsol_f_nrc1}
\sigma_{t,f}^2 = &\frac{f(t)}{f_g^W((P_g^W)^{-1}(F(t)))} \mathbf{1}_{\{0 < F(t) < 1\}} & \text{ almost everywhere for } t \geq 0 .
\end{eqnarray}
\end{definition}
\noindent We introduce the notation $h(t)=\operatorname{erfinv}\left\{1-F(t)\right\}$. If we substitute $(P_g^{W})^{-1}$ in Equation (\ref{def_expsol_f_nrc1}) with Equation (\ref{PgWinvt}) from Lemma \ref{lemma_inv_nrc1}, we can reexpress the explicit solution as 
\begin{eqnarray}
\label{volsimp0}
\sigma_{t,f}^2 = & \frac{f(t)}{\frac{2}{g^2\sqrt{\pi}}{h(t)}^3e^{-{h(t)}^2}} \mathbf{1}_{\{0 < F(t) < 1\}} & \text{ for } t \geq 0.
\end{eqnarray}
Let us give a set of assumptions sufficient for $Z^{f}$ to satisfy Assumption \ref{assDDS_nrc1}.
\begin{assumption}
\label{assmain_f_nrc1} We assume that there exists $\eta_F^0 > 0$ s.t. the explicit solution is locally integrable on $[K_F^0,K_F^0+\eta_F^0]$, i.e.,
\begin{eqnarray}
\label{assumptionvol1}
\sigma_{f}^2 \restriction_{[K_F^0, K_F^0 + \eta_F^0]} \in L_{1,loc}\big([K_F^0, K_F^0 + \eta_F^0]\big).
\end{eqnarray}
Moreover, we assume that $K_F^1$ is not finite.
\end{assumption}
\noindent The following theorem in the particular case when the quadratic variation $\langle Z\rangle$ is absolutely continuous states that under Assumption \ref{assmain_f_nrc1}, (a) $Z^f$ satisfies Assumption \ref{assDDS_nrc1} and (b) that any variance function is solution if and only if it is an explicit solution. The proof of (a) is mainly based on topological argument in $\reels^+$ and the use of Assumption \ref{assmain_f_nrc1}. The proof of (b) is based on substituting the left-hand side of Equation (\ref{PgZf_nrc1}) with Equations (\ref{eq_theorem_nrc1}) from Theorem \ref{theorem_nrc1} and (\ref{sigma_f_def_nrc1}) and then differentiating and inverting on both sides of the equation to derive the explicit solution.
\begin{theorem}
\label{theorem_f_nrc1}
Under Assumption \ref{assmain_f_nrc1}, (a) $Z^f$ satisfies Assumption \ref{assDDS_nrc1} (b) (i) $\sigma_f^2$ is a solution in the sense of Definition \ref{def_solution_f_nrc1} $\iff$ (ii) $\sigma_f^2$ is an explicit solution in the sense of Definition \ref{def_explicitsolution_f_nrc1}.
\end{theorem}

\subsection{Two-sided (i) case}
The first lemma shows that there exists an invert of $P_{g,h}^W$ which we denote $(P_{g,h}^W)^{-1}$ and is strictly increasing such that 
$(P_{g,h}^{W})^{-1}(0)=0$ and $(P_{g,h}^{W})^{-1}(1)=\infty$, all of which are new results which will be useful to prove the explicit solution of the inverse problem. The proof relies on Lemma \ref{lemma_nrc2}.
\begin{lemma}
\label{lemma_inv_nrc2} There exists an invert of $P_{g,h}^W$ which we denote $(P_{g,h}^W)^{-1} : [0,1) \rightarrow \reels^+$ and is strictly increasing such that 
$(P_{g,h}^{W})^{-1}(0)=0$ and $(P_{g,h}^{W})^{-1}(1)=\infty$.
\end{lemma}
\noindent By Equation (\ref{eq_theorem_nrc2}) from Theorem \ref{theorem_nrc2}, we have that $P_{g,h}^Z (t)  = P_{g,h}^W \big(\langle Z\rangle_{t}\big) \text{ for } t \geq 0$.
We first give the definition of a solution in the inverse first-passage time problem of quadratic variation.
\begin{definition}
\label{def_solution_F_nrc2}
For any cdf $F$, we say that a nonrandom nondecreasing function $
\tilde{v}_{F} : \reels^+ \rightarrow \reels^+ $ which is the quadratic variation of a continuous local martingale $Z^F$, i.e.,
\begin{eqnarray}
\label{tildev_F_def_nrc2}
\langle Z^F\rangle_{t} &=& \tilde{v}_{F}(t) \text{ for } t \geq 0,
\end{eqnarray}
is solution if it satisfies
\begin{eqnarray}
\label{PgZF_nrc2}
P_{g,h}^{Z^F} (t) &=& F(t) \text{ for } t \geq 0.
\end{eqnarray}
\end{definition}
\noindent We now give the definition of the explicit solution.
\begin{definition}
\label{def_explicitsolution_F_nrc2}
For any cdf $F$, we say that a nonrandom nondecreasing function $v_F$ which is the quadratic variation of a continuous local martingale $Z^F$, i.e.,
\begin{eqnarray}
\label{v_F_def_nrc2}
\langle Z^F\rangle_{t} &=& v_{F}(t) \text{ for } t \geq 0,
\end{eqnarray} 
is an explicit solution if it is of the form
\begin{eqnarray}
\label{def_expsol_F_nrc2}
v_{F}(t) = &(P_{g,h}^W)^{-1}(F(t)) \mathbf{1}_{\{0 < F(t) < 1\}} & \text{ for } t \geq 0.
\end{eqnarray}
\end{definition}
\noindent The following theorem states that under Assumption \ref{assmain_F_nrc1}, (a) $Z^F$ satisfies Assumption \ref{assDDS_nrc1} and (b) that any nondecreasing function is solution if and only if it is an explicit solution. The proof of (b) is based on substituting the left-hand side of Equation (\ref{PgZF_nrc2}) with Equations (\ref{eq_theorem_nrc2}) and (\ref{v_F_def_nrc2}) and then inverting on both sides of the equation to derive the explicit solution.
\begin{theorem}
\label{theorem_F_nrc2}
Under Assumption \ref{assmain_F_nrc1}, (a) $Z^{F}$ satisfies Assumption \ref{assDDS_nrc1} (b) (i) $v_F$ is a solution in the sense of Definition \ref{def_solution_F_nrc2} $\iff$ (ii) $v_F$ is an explicit solution in the sense of Definition \ref{def_explicitsolution_F_nrc2}.
\end{theorem}
\noindent In what follows, we consider the particular case when the quadratic variation $\langle Z\rangle$ is absolutely continuous. We give the definition of a solution in the inverse first-passage time problem of absolutely continuous quadratic variation. Since the quadratic variation is absolutely continuous, we can consider its derivative in the solution definition. 
\begin{definition}
\label{def_solution_f_nrc2}
For any pdf $f$, we say that a variance function $\widetilde{\sigma}_{f}^2 : \reels^+ \rightarrow \reels^+ $ which is the quadratic variation  derivative of a continuous local martingale $Z^f$, i.e.,
\begin{eqnarray}
\label{tildesigma_f_def_nrc2}
\langle Z^f\rangle_{t} &=& \int_0^t \widetilde{\sigma}_{s,f}^2 ds \text{ for } t \geq 0,
\end{eqnarray}
is solution if it satisfies
\begin{eqnarray}
\label{PgZf_nrc2}
P_{g,h}^{Z^f} (t) &=& F(t) \text{ for } t \geq 0.
\end{eqnarray}
\end{definition}
\noindent We now give the definition of the explicit solution.
\begin{definition}
\label{def_explicitsolution_f_nrc2}
For any pdf $f$, we say that a variance function $\sigma_{f}^2 : \reels^+ \rightarrow \reels^+ $ which is the quadratic variation  derivative of a continuous local martingale $Z^f$, i.e.,
\begin{eqnarray}
\label{sigma_f_def_nrc2}
\langle Z^f\rangle_{t} &=& \int_0^t \sigma_{s,f}^2 ds \text{ for } t \geq 0,
\end{eqnarray} 
is an explicit solution if it is of the form
\begin{eqnarray}
\label{def_expsol_f_nrc2}
\sigma_{t,f}^2 = &\frac{f(t)}{f_{g,h}^W((P_{g,h}^W)^{-1}(F(t)))} \mathbf{1}_{\{0 < F(t) < 1\}} & \text{ almost everywhere for } t \geq 0 .
\end{eqnarray}
\end{definition}
\noindent Let us give a set of assumptions sufficient for $Z^{f}$ to satisfy Assumption \ref{assDDS_nrc1}.
\begin{assumption}
\label{assmain_f_nrc2} We assume that the explicit solution is locally integrable in $K_F^0$, i.e., there exists $\eta_F^0 > 0$ s.t.
\begin{eqnarray}
\label{assumptionvol1_nrc2}
\sigma_{f}^2 \restriction_{[K_F^0, K_F^0 + \eta_F^0]} \in L_{1}\big([K_F^0, K_F^0 + \eta_F^0]\big).
\end{eqnarray}
Moreover, we also assume that $K_F^1$ is not finite.
\end{assumption}
\noindent The following theorem in the particular case when the quadratic variation $\langle Z\rangle$ is absolutely continuous states that under Assumption \ref{assmain_f_nrc2}, (a) $Z^f$ satisfies Assumption \ref{assDDS_nrc1} and (b) that any variance function is solution if and only if it is an explicit solution. The proof of (a) is mainly based on topological argument in $\reels^+$ and the use of Assumption \ref{assmain_f_nrc2}. The proof of (b) is based on substituting the left-hand side of Equation (\ref{PgZf_nrc2}) with Equations (\ref{eq_theorem_nrc2}) from Theorem \ref{theorem_nrc2} and (\ref{sigma_f_def_nrc2}) and then differentiating and inverting on both sides of the equation to derive the explicit solution.
\begin{theorem}
\label{theorem_f_nrc2}
Under Assumption \ref{assmain_f_nrc2}, (a) $Z^f$ satisfies Assumption \ref{assDDS_nrc1} (b) (i) $\sigma_f^2$ is a solution in the sense of Definition \ref{def_solution_f_nrc2} $\iff$ (ii) $\sigma_f^2$ is an explicit solution in the sense of Definition \ref{def_explicitsolution_f_nrc2}.
\end{theorem}

\subsection{One-sided (ii) case} 
\noindent To define the inverse FPT problem, we introduce the set of random survival cdfs. Since the stochastic process $Y$ has its quadratic variation $\langle Y\rangle$ which is continuous and random, we accordingly consider the set of random continuous cdfs.
\begin{definition} \label{Fdef_rc1}
A function $F:\reels^+ \times \Omega \rightarrow  [0,1]$ is a random survival cdf if $F(\omega)$ is nondecreasing, continuous, 
and satisfies $F(0,\omega) = 0$ and $\underset{t \rightarrow \infty}{\lim} F(t,\omega) =1$ for $\omega \in \Omega$. 
\end{definition}
\noindent By Equation (\ref{eq_theorem_rc1}) from Theorem \ref{theorem_rc1}, we have $P_1^Y (t)  =  \int_{\mathcal{I}(\reels^+,\reels^+)} P_1^W \big(y_t\big) dF_{\langle Y\rangle}(y) \text{ for } t \geq 0$. We first give the definition of a solution in the inverse first-passage time problem of quadratic variation. For any $\mathcal{F}$-adapted continuous stochastic process $(Z_t)_{t \in \reels^+}$ started at 0, i.e., with $Z_0=0$, we define the regular conditional cdf of $\Tau_g^Z$ as $P_g^Z(|\omega)$.
\begin{definition}
\label{def_solution_F_rc1}
For any random cdf $F$, we say that a  nondecreasing stochastic process $
\tilde{v}_{F} : \reels^+ \times \Omega \rightarrow \reels^+ $ which is the quadratic variation of a continuous local martingale $Y^F$, i.e.,
\begin{eqnarray}
\label{tildev_F_def_rc1}
\langle Y^F\rangle_{t} (\omega) &=& \tilde{v}_{F}(t,\omega) \text{ for } t \geq 0 \text{ and } \omega \in \Omega,
\end{eqnarray}
is solution if it satisfies
\begin{eqnarray}
\label{PgZF_rc1}
P_1^{Y^F} (t | \omega) &=& F(t,\omega) \text{ for } t \geq 0 \text{ and } \omega \in \Omega.
\end{eqnarray}
\end{definition}
\noindent Equation (\ref{tildev_F_def_rc1}) in Definition \ref{def_solution_F_rc1} implicitly implies the existence of a continuous local martingale with quadratic variation $\tilde{v}_{F}$, but this is true since any standard Wiener process with random time-changed $\frac{\tilde{v}_{F}(t)}{t}$ will have $\tilde{v}_{F}$ as quadratic variation. We now give the definition of the explicit solution.
\begin{definition}
\label{def_explicitsolution_F_rc1}
For any random cdf $F$, we say that a nondecreasing stochastic process $v_F$ which is the quadratic variation of a continuous local martingale $Y^F$, i.e.,
\begin{eqnarray}
\label{v_F_def_rc1}
\langle Y^F\rangle_{t}(\omega) &=& v_{F}(t,\omega) \text{ for } t \geq 0 \text{ and } \omega \in \Omega,
\end{eqnarray} 
is an explicit solution if it is of the form
\begin{eqnarray}
\label{def_expsol_F_rc1}
v_{F}(t,\omega) = &(P_1^W)^{-1}(F(t,\omega)) \mathbf{1}_{\{0 < F(t,\omega) < 1\}} & \text{ for } t \geq 0 \text{ and } \omega \in \Omega.
\end{eqnarray}
\end{definition}
\noindent If we substitute $(P_1^{W})^{-1}$ in Equation (\ref{def_expsol_F_rc1}) with Equation (\ref{PgWinvt}) from Lemma \ref{lemma_inv_nrc1}, we can reexpress the explicit solution as 
\begin{eqnarray}
\label{expsoldefsimp_rc1}
v_{F}(t,\omega) = & \frac{1}{2\operatorname{erfinv}(1-F(t,\omega))^2} \mathbf{1}_{\{0 < F(t,\omega) < 1\}} & \text{ for } t \geq 0 \text{ and } \omega \in \Omega.
\end{eqnarray}
We define the infimum value such that $F(t)$ is positive and the infimum value such that $F(t)$ equals unity as
\begin{eqnarray}
 \label{defKF0_rc1}
 K_F^0 (\omega) & = & \inf \{ t > 0 \text{ such that } F(t,\omega)>0 \} \text{ and }\\
 \label{defKF1_rc1}
 K_F^1(\omega) & = & \inf \{ t > 0 \text{ such that } F(t,\omega)=1 \}.
\end{eqnarray}
Let us give an assumption sufficient for $Y^{F}$ to satisfy Assumption \ref{assDDS_rc1}.
\begin{assumption}
\label{assmain_F_rc1} We assume that $K_F^1(\omega)$ is not finite for $\omega \in \Omega$.
\end{assumption}
\noindent The following theorem states that under Assumption \ref{assmain_F_rc1}, (a) $Y^F$ satisfies Assumption \ref{assDDS_rc1} and (b) that any random nondecreasing function is solution if and only if it is an explicit solution. The proof of (b) is based on substituting the left-hand side of Equation (\ref{PgZF_rc1}) with Equations (\ref{eq_theorem_rc1}) and (\ref{v_F_def_rc1}) and then inverting on both sides of the equation to derive the explicit solution.
\begin{theorem}
\label{theorem_F_rc1}
Under Assumption \ref{assmain_F_rc1}, (a) $Y^{F}$ satisfies Assumption \ref{assDDS_rc1} (b) (i) $v_F$ is a solution in the sense of Definition \ref{def_solution_F_rc1} $\iff$ (ii) $v_F$ is an explicit solution in the sense of Definition \ref{def_explicitsolution_F_rc1}.
\end{theorem}
\noindent In what follows, we consider the particular case when the quadratic variation $\langle Y\rangle$ is absolutely continuous.  We accordingly consider the set of random absolutely continuous cdfs.
\begin{definition} \label{fdef_rc1}
A function $f:\reels^+ \times \Omega \rightarrow \reels^+$ is a random survival pdf if it satifies
\begin{eqnarray}
\label{Ffrelation_rc1}
F(t,\omega) & = &\int_0^t f(s,\omega) ds \text{ for } t \geq 0 \text{ and } \omega \in \Omega.
\end{eqnarray}
\end{definition}
\noindent We give the definition of a solution in the inverse first-passage time problem of absolutely continuous quadratic variation. Since the quadratic variation is absolutely continuous, we can consider its derivative in the solution definition.
\begin{definition}
\label{def_solution_f_rc1}
For any random pdf $f$, we say that a variance stochastic process $\widetilde{\sigma}_{f}^2 : \reels^+ \times \Omega \rightarrow \reels^+ $ which is the quadratic variation  derivative of a continuous local martingale $Y^f$, i.e.,
\begin{eqnarray}
\label{tildesigma_f_def_rc1}
\langle Y^f\rangle_{t} (\omega) &=& \int_0^t \widetilde{\sigma}_{s,f}^2 (\omega) ds \text{ for } t \geq 0 \text{ and } \omega \in \Omega,
\end{eqnarray}
is solution if it satisfies
\begin{eqnarray}
\label{PgZf_rc1}
P_1^{Y^F} (t | \omega)  &=& F(t,\omega) \text{ for } t \geq 0 \text{ and } \omega \in \Omega.
\end{eqnarray}
\end{definition}
\noindent Equation (\ref{tildesigma_f_def_rc1}) in Definition \ref{def_solution_f_rc1} implicitly implies the existence of a continuous local martingale with random quadratic variation $\int_0^t \widetilde{\sigma}_{s,f}^2 ds$, but this is true since we can consider It\^{o} processes from Example \ref{example_ito_rc1}. We now give the definition of the explicit solution.
\begin{definition}
\label{def_explicitsolution_f_rc1}
For any random pdf $f$, we say that a variance stochastic process $\sigma_{f}^2 : \reels^+ \times \Omega \rightarrow \reels^+ $ which is the quadratic variation  derivative of a continuous local martingale $Y^f$, i.e.,
\begin{eqnarray}
\label{sigma_f_def_rc1}
\langle Y^f\rangle_{t} (\omega) &=& \int_0^t \sigma_{s,f}^2(\omega) ds \text{ for } t \geq 0 \text{ and } \omega \in \Omega,
\end{eqnarray} 
is an explicit solution if it is of the form
\begin{eqnarray}
\label{def_expsol_f_rc1}
\sigma_{t,f}^2(\omega) = \frac{f(t,\omega)}{f_1^W((P_1^W)^{-1}(F(t,\omega)))} \mathbf{1}_{\{0 < F(t,\omega) < 1\}} \\ \nonumber   \text{ almost everywhere for } t \geq 0 \text{ and } \omega \in \Omega.
\end{eqnarray}
\end{definition}
\noindent We introduce the notation $h(t,\omega)=\operatorname{erfinv}\left\{1-F(t,\omega)\right\}$. If we substitute $(P_1^{W})^{-1}$ in Equation (\ref{def_expsol_f_rc1}) with Equation (\ref{PgWinvt}) from Lemma \ref{lemma_inv_nrc1}, we can reexpress the explicit solution as 
\begin{eqnarray}
\label{volsimp0_rc1}
\sigma_{t,f}^2(\omega) = & \frac{f(t,\omega)}{\frac{2}{g^2\sqrt{\pi}}{h(t,\omega)}^3e^{-{h(t,\omega)}^2}} \mathbf{1}_{\{0 < F(t,\omega) < 1\}} & \text{ for } t \geq 0 \text{ and } \omega \in \Omega.
\end{eqnarray}
Let us give a set of assumptions sufficient for $Y^{f}$ to satisfy Assumption \ref{assDDS_rc1}.
\begin{assumption}
\label{assmain_f_rc1} We assume that the explicit solution is locally integrable on $\reels^+ \times \Omega$, i.e., 
\begin{eqnarray}
\label{assumptionvol1_rc1}
\sigma_{f}^2 \in L_{1,loc}\big(\reels^+ \times \Omega \big).
\end{eqnarray}
Moreover, we also assume that $K_F^1$ is not finite.
\end{assumption}
\noindent The following theorem in the particular case when the quadratic variation $\langle Y\rangle$ is absolutely continuous states that under Assumption \ref{assmain_f_rc1}, (a) $Y^f$ satisfies Assumption \ref{assDDS_rc1} and (b) that any variance function is solution if and only if it is an explicit solution. The proof of (a) is mainly based on the use of Assumption \ref{assmain_f_rc1}. The proof of (b) is based on substituting the left-hand side of Equation (\ref{PgZf_rc1}) with Equations (\ref{eq_theorem_rc1}) from Theorem \ref{theorem_rc1} and (\ref{sigma_f_def_rc1}) and then differentiating and inverting on both sides of the equation to derive the explicit solution.
\begin{theorem}
\label{theorem_f_rc1}
Under Assumption \ref{assmain_f_rc1}, (a) $Y^f$ satisfies Assumption \ref{assDDS_rc1} (b) (i) $\sigma_f^2$ is a solution in the sense of Definition \ref{def_solution_f_rc1} $\iff$ (ii) $\sigma_f^2$ is an explicit solution in the sense of Definition \ref{def_explicitsolution_f_rc1}.
\end{theorem}

\subsection{Two-sided (ii) case} 
\noindent By Equation (\ref{eq_theorem_rc2}) from Theorem \ref{theorem_rc2}, we have $P_{g,h}^Z (t)  =  \int_{\mathcal{S}} P_{g_0,h_0}^W \big(z_t\big) dF_{u}(g_0,h_0,z) \text{ for } t \geq 0$. We first give the definition of a solution in the inverse first-passage time problem of quadratic variation. For any $\mathcal{F}$-adapted continuous stochastic process $(Z_t)_{t \in \reels^+}$ started at 0, i.e., with $Z_0=0$, we define the regular conditional cdf of $\Tau_{g,h}^Z$ as $P_{g,h}^Z(|\omega)$.
\begin{definition}
\label{def_solution_F_rc2}
For any random cdf $F$, we say that a  nondecreasing stochastic process $
\tilde{v}_{F} : \reels^+ \times \Omega \rightarrow \reels^+ $ which is the quadratic variation of a continuous local martingale $Z^F$, i.e.,
\begin{eqnarray}
\label{tildev_F_def_rc2}
\langle Z^F\rangle_{t} (\omega) &=& \tilde{v}_{F}(t,\omega) \text{ for } t \geq 0 \text{ and } \omega \in \Omega,
\end{eqnarray}
is solution if it satisfies
\begin{eqnarray}
\label{PgZF_rc2}
P_{g,h}^{Z^F} (t | \omega) &=& F(t,\omega) \text{ for } t \geq 0 \text{ and } \omega \in \Omega.
\end{eqnarray}
\end{definition}
\noindent Equation (\ref{tildev_F_def_rc2}) in Definition \ref{def_solution_F_rc2} implicitly implies the existence of a continuous local martingale with quadratic variation $\tilde{v}_{F}$, but this is true since any standard Wiener process with random time-changed $\frac{\tilde{v}_{F}(t)}{t}$ will have $\tilde{v}_{F}$ as quadratic variation. We now give the definition of the explicit solution.
\begin{definition}
\label{def_explicitsolution_F_rc2}
For any random cdf $F$, we say that a nondecreasing stochastic process $v_F$ which is the quadratic variation of a continuous local martingale $Z^F$, i.e.,
\begin{eqnarray}
\label{v_F_def_rc2}
\langle Z^F\rangle_{t}(\omega) &=& v_{F}(t,\omega) \text{ for } t \geq 0 \text{ and } \omega \in \Omega,
\end{eqnarray} 
is an explicit solution if it is of the form
\begin{eqnarray}
\label{def_expsol_F_rc2}
v_{F}(t,\omega) = &(P_{g,h}^W)^{-1}(F(t,\omega)) \mathbf{1}_{\{0 < F(t,\omega) < 1\}} & \text{ for } t \geq 0 \text{ and } \omega \in \Omega.
\end{eqnarray}
\end{definition}
\noindent The following theorem states that under Assumption \ref{assmain_F_rc1}, (a) $Z^F$ satisfies Assumption \ref{assDDS_rc1} and (b) that any random nondecreasing function is solution if and only if it is an explicit solution. The proof of (b) is based on substituting the left-hand side of Equation (\ref{PgZF_rc2}) with Equations (\ref{eq_theorem_rc2}) and (\ref{v_F_def_rc2}) and then inverting on both sides of the equation to derive the explicit solution.
\begin{theorem}
\label{theorem_F_rc2}
Under Assumption \ref{assmain_F_rc1}, (a) $Z^{F}$ satisfies Assumption \ref{assDDS_rc1} (b) (i) $v_F$ is a solution in the sense of Definition \ref{def_solution_F_rc2} $\iff$ (ii) $v_F$ is an explicit solution in the sense of Definition \ref{def_explicitsolution_F_rc2}.
\end{theorem}
\noindent In what follows, we consider the particular case when the quadratic variation $\langle Z\rangle$ is absolutely continuous. We give the definition of a solution in the inverse first-passage time problem of absolutely continuous quadratic variation. Since the quadratic variation is absolutely continuous, we can consider its derivative in the solution definition.
\begin{definition}
\label{def_solution_f_rc2}
For any random pdf $f$, we say that a variance stochastic process $\widetilde{\sigma}_{f}^2 : \reels^+ \times \Omega \rightarrow \reels^+ $ which is the quadratic variation  derivative of a continuous local martingale $Z^f$, i.e.,
\begin{eqnarray}
\label{tildesigma_f_def_rc2}
\langle Z^f\rangle_{t} (\omega) &=& \int_0^t \widetilde{\sigma}_{s,f}^2 (\omega) ds \text{ for } t \geq 0 \text{ and } \omega \in \Omega,
\end{eqnarray}
is solution if it satisfies
\begin{eqnarray}
\label{PgZf_rc2}
P_{g,h}^{Z^F} (t | \omega)  &=& F(t,\omega) \text{ for } t \geq 0 \text{ and } \omega \in \Omega.
\end{eqnarray}
\end{definition}
\noindent Equation (\ref{tildesigma_f_def_rc2}) in Definition \ref{def_solution_f_rc2} implicitly implies the existence of a continuous local martingale with random quadratic variation $\int_0^t \widetilde{\sigma}_{s,f}^2 ds$, but this is true since we can consider It\^{o} processes from Example \ref{example_ito_rc1}. We now give the definition of the explicit solution.
\begin{definition}
\label{def_explicitsolution_f_rc2}
For any random pdf $f$, we say that a variance stochastic process $\sigma_{f}^2 : \reels^+ \times \Omega \rightarrow \reels^+ $ which is the quadratic variation  derivative of a continuous local martingale $Z^f$, i.e.,
\begin{eqnarray}
\label{sigma_f_def_rc2}
\langle Z^f\rangle_{t} (\omega) &=& \int_0^t \sigma_{s,f}^2(\omega) ds \text{ for } t \geq 0 \text{ and } \omega \in \Omega,
\end{eqnarray} 
is an explicit solution if it is of the form
\begin{eqnarray}
\label{def_expsol_f_rc2}
\sigma_{t,f}^2(\omega) = \frac{f(t,\omega)}{f_{g,h}^W((P_{g,h}^W)^{-1}(F(t,\omega)))} \mathbf{1}_{\{0 < F(t,\omega) < 1\}} \\ \nonumber   \text{ almost everywhere for } t \geq 0 \text{ and } \omega \in \Omega.
\end{eqnarray}
\end{definition}
\noindent The following theorem in the particular case when the quadratic variation $\langle Z\rangle$ is absolutely continuous states that under Assumption \ref{assmain_f_rc1}, (a) $Z^f$ satisfies Assumption \ref{assDDS_rc1} and (b) that any variance function is solution if and only if it is an explicit solution. The proof of (a) is mainly based on the use of Assumption \ref{assmain_f_rc1}. The proof of (b) is based on substituting the left-hand side of Equation (\ref{PgZf_rc2}) with Equations (\ref{eq_theorem_rc2}) from Theorem \ref{theorem_rc2} and (\ref{sigma_f_def_rc2}) and then differentiating and inverting on both sides of the equation to derive the explicit solution.
\begin{theorem}
\label{theorem_f_rc2}
Under Assumption \ref{assmain_f_rc1}, (a) $Z^f$ satisfies Assumption \ref{assDDS_rc1} (b) (i) $\sigma_f^2$ is a solution in the sense of Definition \ref{def_solution_f_rc2} $\iff$ (ii) $\sigma_f^2$ is an explicit solution in the sense of Definition \ref{def_explicitsolution_f_rc2}.
\end{theorem}

\section{Proofs of the explicit formula}
\label{section_proofs_expfor}
In this section, we prove the explicit formula for the one-sided and two-sided boundary crossing probability (\ref{def_bcp})-(\ref{def_bcp2}) in the case (i) and (ii).
\subsection{One-sided (i) case}
We start with the proof of Lemma \ref{lemma_nrc1}, which are well-known results from \cite{malmquist1954certain} (Theorem 1, p. 526).
\begin{proof}[Proof of Lemma \ref{lemma_nrc1}]
By \cite{malmquist1954certain} (Theorem 1, p. 526), we have that the probability that a standard Wiener process crosses a constant boundary $g$ conditioned on its arrival value $x$ at arrival time $T$ is given by
\begin{eqnarray}
\label{malmquisteq0}
\proba (\Tau^Z_g \leq T | W_T = x) = \exp{\Big(-\frac{2g(g-x)}{T}\Big)}\mathbf{1}_{\{x \leq g\}} + \mathbf{1}_{\{x > g\}}
\end{eqnarray}
for any $x \in \reels$. \cite{wang1997boundary} (Equations (3), p. 55) integrate Equation (\ref{malmquisteq0}) with respect to the Wiener process arrival value $s$ and derive the cdf as $P_g^W (0)  =  0$ and Equation (\ref{eqP_lemma_nrc1}). Then, we can deduce the pdf for $t>0$ as
\begin{align*}
f_g^W(t)&=\frac{d}{dt} P_g^W(t)\\
&=\frac{d}{dt}\left(1 - \Phi \left(\frac{g}{\sqrt{t}}\right) + \Phi \left(\frac{-g}{\sqrt{t}}\right)\right)\\
&=\frac{d}{dt}\left(1 - \int_{- \infty}^{\frac{g}{\sqrt{t}}} {\displaystyle {\frac {1}{ {\sqrt {2\pi }}}}e^{-{\frac {u^2}{2}}}} du + \int_{- \infty}^{\frac{- g}{\sqrt{t}}} {\displaystyle {\frac {1}{ {\sqrt {2\pi }}}}e^{-{\frac {u^2}{2}}}} du \right)\\ &=\frac{ge^{-\frac{g^2}{2t}}}{\sqrt{2\pi t^3}},
\end{align*}
where we use Equation (\ref{fZgt_nrc1}) in the first equality, Equation (\ref{eqP_lemma_nrc1}) in the second equality, Equation (\ref{def_standardgaussian}) in the third equality, the fundamental theorem of calculus along with the chain rule in the fourth equality. We have thus shown Equation (\ref{eqf_lemma_nrc1}). 

%We consider now the upper and lower boundary case. In view of \cite{anderson1960modification} (Theorem 5.1, p. 191), we can deduce Equation (\ref{fgWex21}). We provide the derivation of $P_g^W(t)$ in what follows. In order to integrate $f_g^W(t)$, i.e., Equation (\ref{fgWex21}), we first derive the integral of $\operatorname{ss}_t(v,w)$ for any $0<v<w$ as
%\begin{eqnarray}
%\nonumber \int_0^t{\mathrm{ss}_{x}(v, w)dx}&=&\sum_{k=-\infty}^{\infty}\frac{w-v+2kw}{\sqrt{2\pi}}\int_0^t{x^{-3/2}e^{-(w-v+2kw)^{2}/2x}dx}\\ \label{intssx}
%&=&\sum_{k=-\infty}^{\infty}\left(1-\operatorname{erf}\left(\frac{w-v+2kw}{\sqrt{2t}}\right)\right),
%\end{eqnarray}
%where we use Equation (\ref{sst}) in the first equality and Equation (\ref{erf}) in the second equality. 
%Then, we can obtain $P_g^W (0)  = 0$ and $P_g^W(t)$ by integrating Equation (\ref{fgWex21}) with the use of Equation (\ref{intssx}) for any $t>0$ in the first equality and algebraic manipulation in the second equality:
%\begin{align*}
%P_g^W (t) &= \sum_{k=-\infty}^{\infty}{\left(1-\operatorname{erf}\left(\frac{-g^{(1)}+2k(g-g^{(1)})}{\sqrt{2t}}\right)\right)}+\sum_{k=-\infty}^{\infty}{\left(1-\operatorname{erf}\left(\frac{g+2k(g-g^{(1)})}{\sqrt{2t}}\right)\right)}\\
%&=\sum_{k=-\infty}^{\infty}{\left(2-\operatorname{erf}\left(\frac{-g^{(1)}+2k(g-g^{(1)})}{\sqrt{2t}}\right)-\operatorname{erf}\left(\frac{g+2k(g-g^{(1)})}{\sqrt{2t}}\right)\right)}.
%\end{align*}
\end{proof}
\noindent We give now the proof of Proposition \ref{lemmaDDS_nrc1}. The main idea to prove it is the scale invariant property of the time-changed Wiener process and thus the scale invariant property of the FPT. 
\begin{proof}[Proof of Proposition \ref{lemmaDDS_nrc1}]
We have that for any $t \geq 0$
\begin{eqnarray}
\nonumber \big\{ \Tau_g^Z =t \big\} 
& = & \big\{ \inf \{s \geq 0 \text{ s.t. } Z_s \geq g\} = t \big\}\\ \label{eq_proof_lemmaDDS_nrc1} 
& = & \big\{ \inf \{s \geq 0\text{ s.t. } B_{\langle Z\rangle_{s}} \geq g\} = t \big\},
\end{eqnarray}
where we use Equation (\ref{TgZdef_nrc1}) in the first equality, and Equation (\ref{proof230110}) in the second equality. Since $B$ is a $({\mathcal {F}}_{\langle Z\rangle_{t}^{-1}})_{t\geq 0}$-Wiener process, $W$ is an $({\mathcal {F}}_{t})_{t\geq 0}$-Wiener process and the boundary is constant, we can use a time change of quadratic variation $\langle Z\rangle_t$ value to obtain that
\begin{eqnarray}
\label{eq_proof_lemmaDDS_nrc2} \big\{ \inf \{s \geq 0\text{ s.t. } B_{\langle Z\rangle_{s}} \geq g\} = t \big\} = \big\{\inf \{s \geq 0 \text{ s.t. } B_s \geq g \} = \langle Z\rangle_t \big\}.
\end{eqnarray}
Then, we can calculate by Equation (\ref{TgZdef_nrc1}) that
\begin{eqnarray}
\label{eq_proof_lemmaDDS_nrc3} \big\{\inf \{s \geq 0 \text{ s.t. } B_s \geq g \} = \langle Z\rangle_t \big\} & = & \big\{ \Tau^B_g = \langle Z\rangle_t \big\}, 
\end{eqnarray}
By Equations (\ref{eq_proof_lemmaDDS_nrc1}), (\ref{eq_proof_lemmaDDS_nrc2}) and (\ref{eq_proof_lemmaDDS_nrc3}), we can deduce Equation (\ref{eq_lemmaDDS_nrc1}).
\end{proof}
\noindent In what follows, we give the proof of Theorem \ref{theorem_nrc1}. The proof is mainly based on Proposition \ref{lemmaDDS_nrc1}.
\begin{proof}[Proof of Theorem \ref{theorem_nrc1}]
We have that for any $t \geq 0$
\begin{eqnarray}
\nonumber P_g^Z(t) & = & \proba (\Tau^Z_g \leq t)\\ \nonumber 
& = & \proba \big(\inf \{s \geq 0 \text{ s.t. } Z_s \geq g\} \leq t \big)\\ \label{eq_proof_theorem_nrc1} 
& = & \proba \big(\inf \{s \geq 0\text{ s.t. } B_{\langle Z\rangle_{s}} \geq g\} \leq t\big),
\end{eqnarray}
where we use Equation (\ref{PgZdef_nrc1}) in the first equality, Equation (\ref{TgZdef_nrc1}) in the second equality, and Equation (\ref{proof230110}) in the third equality. By Lemma \ref{lemmaDDS_nrc1} along with Assumption \ref{assDDS_nrc1}, we obtain that
\begin{eqnarray}
\label{eq_proof_theorem_nrc2}\proba \Big(\inf \{s \geq 0 \text{ s.t. } B_{\langle Z\rangle_s}\geq g\} \leq t \Big)=\proba \Big(\inf \{s \geq 0 \text{ s.t. } B_s \geq g \} \leq \langle Z\rangle_t \Big).
\end{eqnarray}
Then, we can calculate that
\begin{eqnarray}
\nonumber \proba \Big(\inf \{s \geq 0 \text{ s.t. } B_s \geq g \} \leq \langle Z\rangle_t \Big)& = & \proba \Big(\inf \{s \geq 0 \text{ s.t. } W_s \geq g \} \leq \langle Z\rangle_t \Big)\\ \nonumber & = &\proba \Big(\Tau^W_g \leq \langle Z\rangle_t\Big),\\ \label{eq_proof_theorem_nrc3}
& = & P_g^W\big(\langle Z\rangle_t\big), 
\end{eqnarray}
where we use the fact that $B$ and $W$ have the same distribution in the first equality, Equation (\ref{TgZdef_nrc1}) in the second equality, and Equation (\ref{PgZdef_nrc1}) in the third equality. By Equations (\ref{eq_proof_theorem_nrc1}), (\ref{eq_proof_theorem_nrc2}) and (\ref{eq_proof_theorem_nrc3}), we can deduce Equation (\ref{eq_theorem_nrc1}).
\end{proof}
\noindent Finally, we give the proof of Corollary \ref{corollary_nrc1}.
\begin{proof}[Proof of Corollary \ref{corollary_nrc1}]
We have for $t \geq 0$
\begin{eqnarray*}
f_g^X (t) & = & \frac{dP_g^X(t)}{dt}\\
& = & \frac{d(P_g^W(\langle Z\rangle_t))}{dt}\\
& = & \langle Z\rangle_t ' f_g^W \big(\langle Z\rangle_t\big),
\end{eqnarray*}
where we use Equation (\ref{fZgt_nrc1}) in the first equality, Equation (\ref{eq_theorem_nrc1}) from Theorem \ref{theorem_nrc1} along with Assumption \ref{assDDS_nrc1} in the second equality, and the fundamental theorem of calculus along with chain rule and the assumption that the quadratic variation $\langle Z\rangle$ is absolutely continuous on $\reels^+$ in the third equality.
\end{proof}

\subsection{Two-sided (i) case}
We start with the proof of Lemma \ref{lemma_nrc2}, which is well-known results from \cite{anderson1960modification} (Theorem 5.1, p. 191).
\begin{proof}[Proof of Lemma \ref{lemma_nrc2}] Equation (\ref{eqf_lemma_nrc2}) is a more compact form of the explicit formulae (4.3)-(4.4) (Theorem 5.1, p. 191) in \cite{anderson1960modification}. Then, we derive the integral of $\operatorname{ss}(v,w)$ for any $0<v<w$ as
\begin{eqnarray}
\nonumber \int_0^t{\mathrm{ss}_{x}(v, w)dx}&=&\sum_{k=-\infty}^{\infty}\frac{w-v+2kw}{\sqrt{2\pi}}\int_0^t{x^{-3/2}e^{-(w-v+2kw)^{2}/2x}dx}\\ \label{intssx}
&=&\sum_{k=-\infty}^{\infty}\left(2-2 \phi\left(\frac{w-v+2kw}{\sqrt{t}}\right)\right),
\end{eqnarray}
where we use Equation (\ref{sst}) in the first equality. Then, we can obtain $P_{g,h}^W (0)  = 0$ and Equation (\ref{eqP_lemma_nrc2}) by integrating Equation (\ref{eqf_lemma_nrc2}) with the use of Equation (\ref{intssx}) for any $t>0$.
\end{proof}

\noindent We give now the proof of Proposition \ref{lemmaDDS_nrc2}. The main idea to prove it is the scale invariant property of the time-changed Wiener process and thus the scale invariant property of the FPT which adapts to the two-sided 
 (i) case since the arguments of boundary constancy and Dambis, Dubins-Schwarz theorem still hold.
\begin{proof}[Proof of Proposition \ref{lemmaDDS_nrc2}]
We have that for any $t \geq 0$
\begin{eqnarray}
\nonumber \big\{ \Tau_{g,h}^Z =t \big\} 
& = & \big\{ \inf \{s \geq 0 \text{ s.t. } Z_s \geq g \text{ or } Z_s \leq h\} = t \big\}\\ \label{eq_proof_lemmaDDS_nrc21} 
& = & \big\{ \inf \{s \geq 0\text{ s.t. } B_{\langle Z\rangle_{s}} \geq g \text{ or } B_{\langle Z\rangle_{s}} \leq h\} = t \big\},
\end{eqnarray}
where we use Equation (\ref{TgZdef_nrc2}) in the first equality, and Equation (\ref{proof230110_nrc2}) in the second equality. Since $B$ is a $({\mathcal {F}}_{\langle Z\rangle_{t}^{-1}})_{t\geq 0}$-Wiener process, $W$ is an $({\mathcal {F}}_{t})_{t\geq 0}$-Wiener process and the two-sided boundary is constant, we can use a time change of quadratic variation $\langle Z\rangle_t$ value to obtain that
\begin{eqnarray}
\label{eq_proof_lemmaDDS_nrc22} \big\{ \inf \{s \geq 0\text{ s.t. } B_{\langle Z\rangle_{s}} \geq g \text{ or } B_{\langle Z\rangle_{s}} \leq h\} = t \big\} \\ \nonumber = \big\{\inf \{s \geq 0 \text{ s.t. } B_s \geq g \text{ or } B_{s} \leq h \} = \langle Z\rangle_t \big\}.
\end{eqnarray}
Then, we can calculate by Equation (\ref{TgZdef_nrc2}) that
\begin{eqnarray}
\label{eq_proof_lemmaDDS_nrc23} \big\{\inf \{s \geq 0 \text{ s.t. } B_s \geq g \text{ or } B_{s} \leq h \} = \langle Z\rangle_t \big\} & = & \big\{ \Tau^B_{g,h} = \langle Z\rangle_t \big\}. 
\end{eqnarray}
By Equations (\ref{eq_proof_lemmaDDS_nrc21}), (\ref{eq_proof_lemmaDDS_nrc22}) and (\ref{eq_proof_lemmaDDS_nrc23}), we can deduce Equation (\ref{eq_lemmaDDS_nrc2}).
\end{proof}
\noindent In what follows, we give the proof of Theorem \ref{theorem_nrc2}. The proof is mainly based on Proposition \ref{lemmaDDS_nrc2}.
\begin{proof}[Proof of Theorem \ref{theorem_nrc2}] We have that for any $t \geq 0$
\begin{eqnarray}
\nonumber P_{g,h}^Z(t) & = & \proba (\Tau^Z_{g,h} \leq t)\\ \nonumber 
& = & \proba \big(\inf \{s \geq 0 \text{ s.t. } Z_s \geq g \text{ or } Z_s \leq h\} \leq t \big)\\ \label{eq_proof_theorem_nrc21} 
& = & \proba \big(\inf \{s \geq 0\text{ s.t. } B_{\langle Z\rangle_{s}} \geq g \text{ or } B_{\langle Z\rangle_{s}} \leq h\} \leq t\big),
\end{eqnarray}
where we use Equation (\ref{PgZdef_nrc2}) in the first equality, Equation (\ref{TgZdef_nrc2}) in the second equality, and Equation (\ref{proof230110_nrc2}) in the third equality. By Lemma \ref{lemmaDDS_nrc2} along with Assumption \ref{assDDS_nrc1}, we obtain that
\begin{eqnarray}
\label{eq_proof_theorem_nrc22}\proba \Big(\inf \{s \geq 0 \text{ s.t. } B_{\langle Z\rangle_s}\geq g \text{ or } B_{\langle Z\rangle_{s}} \leq h\} \leq t \Big) \\ \nonumber =\proba \Big(\inf \{s \geq 0 \text{ s.t. } B_s \geq g \text{ or } B_s \leq h \} \leq \langle Z\rangle_t \Big).
\end{eqnarray}
Then, we can calculate that
\begin{eqnarray}
\nonumber & & \proba \Big(\inf \{s \geq 0 \text{ s.t. } B_s \geq g \text{ or } B_s \leq h \} \leq \langle Z\rangle_t \Big)\\
\nonumber &=& \proba \Big(\inf \{s \geq 0 \text{ s.t. } W_s \geq g \text{ or } W_s \leq h\} \leq \langle Z\rangle_t \Big)\\ \nonumber & = & \proba \Big(\Tau^W_{g,h} \leq \langle Z\rangle_t\Big),\\ \label{eq_proof_theorem_nrc23}
& = & P_{g,h}^W\big(\langle Z\rangle_t\big), 
\end{eqnarray}
where we use the fact that $B$ and $W$ have the same distribution in the second equality, Equation (\ref{TgZdef_nrc2}) in the third equality, and Equation (\ref{PgZdef_nrc2}) in the fourth equality. By Equations (\ref{eq_proof_theorem_nrc21}), (\ref{eq_proof_theorem_nrc22}) and (\ref{eq_proof_theorem_nrc23}), we can deduce Equation (\ref{eq_theorem_nrc2}).
\end{proof}
\noindent Finally, we give the proof of Corollary \ref{corollary_nrc2}.
\begin{proof}[Proof of Corollary \ref{corollary_nrc2}]
We have for $t \geq 0$
\begin{eqnarray*}
f_{g,h}^X (t) & = & \frac{dP_{g,h}^X(t)}{dt}\\
& = & \frac{d(P_{g,h}^W(\langle Z\rangle_t))}{dt}\\
& = & \langle Z\rangle_t ' f_{g,h}^W \big(\langle Z\rangle_t\big),
\end{eqnarray*}
where we use Equation (\ref{fZgt_nrc2}) in the first equality, Equation (\ref{eq_theorem_nrc2}) from Theorem \ref{theorem_nrc2} along with Assumption \ref{assDDS_nrc1} in the second equality, and the fundamental theorem of calculus along with chain rule  and the assumption that the quadratic variation $\langle Z\rangle$ is absolutely continuous on $\reels^+$ in the third equality.
\end{proof}
\subsection{One-sided (ii) case}
\noindent We give now the proof of Proposition \ref{lemmaDDS_rc1}. The main idea to prove it is the scale invariant property of the time-changed Wiener process and thus the scale invariant property of the FPT which adapts to the one-sided 
 (ii) case by using the new process. 
\begin{proof}[Proof of Proposition \ref{lemmaDDS_rc1}]
We have that for any $t \geq 0$
\begin{eqnarray}
\nonumber \big\{ \Tau_1^Y =t \big\} 
& = & \big\{ \inf \{s \geq 0 \text{ s.t. } Y_s \geq 1\} = t \big\}\\ \label{eq_proof_lemmaDDS_rc1} 
& = & \big\{ \inf \{s \geq 0\text{ s.t. } B_{\langle Y\rangle_{s}} \geq 1\} = t \big\},
\end{eqnarray}
where we use Equation (\ref{TgZdef_rc1}) in the first equality, and Equation (\ref{proof230110_rc1}) in the second equality. Since $B$ is a $({\mathcal {F}}_{\langle Z\rangle_{t}^{-1}})_{t\geq 0}$-Wiener process, $W$ is an $({\mathcal {F}}_{t})_{t\geq 0}$-Wiener process and the boundary is constant, we can use a time change of quadratic variation $\langle Z\rangle_t$ value to obtain that
\begin{eqnarray}
\label{eq_proof_lemmaDDS_rc2} \big\{ \inf \{s \geq 0\text{ s.t. } B_{\langle Y\rangle_{s}} \geq 1\} = t \big\} = \big\{\inf \{s \geq 0 \text{ s.t. } B_s \geq 1 \} = \langle Y\rangle_t \big\}.
\end{eqnarray}
Then, we can calculate by Equation (\ref{TgZdef_rc1}) that
\begin{eqnarray}
\label{eq_proof_lemmaDDS_rc3} \big\{\inf \{s \geq 0 \text{ s.t. } B_s \geq 1 \} = \langle Y\rangle_t \big\} & = & \big\{ \Tau^B_1 = \langle Y\rangle_t \big\}. 
\end{eqnarray}
By Equations (\ref{eq_proof_lemmaDDS_rc1}), (\ref{eq_proof_lemmaDDS_rc2}) and (\ref{eq_proof_lemmaDDS_rc3}), we can deduce Equation (\ref{eq_lemmaDDS_rc1}).
\end{proof}
\noindent In what follows, we give the proof of Theorem \ref{theorem_rc1}. The proof is mainly based on Proposition \ref{lemmaDDS_rc1}. We get $P_1^Y$ in the proof of Theorem \ref{theorem_rc1} by integrating $\proba (\Tau^{Y}_1 \leq t |\langle Y\rangle=y)$ with respect to the value of $y \in \mathcal{I}(\reels^+,\reels^+)$.
\begin{proof}[Proof of Theorem \ref{theorem_rc1}]
We have that for any $t \geq 0$
\begin{eqnarray}
\nonumber P_1^Y(t) & = & \proba (\Tau^Y_1 \leq t)\\ \nonumber & = & \int_{\mathcal{I}(\reels^+,\reels^+)} \proba (\Tau^{Y}_1 \leq t |\langle Y\rangle=y) dF_{\langle Y\rangle}(y)\\ \nonumber 
& = & \int_{\mathcal{I}(\reels^+,\reels^+)} \proba \big(\inf \{s \geq 0 \text{ s.t. } Y_s \geq 1\} \leq t |\langle Y\rangle=y\big)dF_{\langle Y\rangle}(y)\\ \label{eq_proof_theorem_rc1} 
& = & \int_{\mathcal{I}(\reels^+,\reels^+)} \proba \big(\inf \{s \geq 0\text{ s.t. } B_{y_{s}} \geq 1\} \leq t\big)dF_{\langle Y\rangle}(y),
\end{eqnarray}
where we use Equation (\ref{PgZdef_rc1}) in the first equality, regular conditional probability in the second equality,  Equation (\ref{TgZdef_rc1}) in the third equality, and Equation (\ref{proof230110_rc1}) in the fourth equality. By Lemma \ref{lemmaDDS_rc1} along with Assumption \ref{assDDS_rc1}, we obtain that
\begin{eqnarray}
\label{eq_proof_theorem_rc2}\int_{\mathcal{I}(\reels^+,\reels^+)} \proba \big(\inf \{s \geq 0\text{ s.t. } B_{y_{s}} \geq 1\} \leq t\big)dF_{\langle Y\rangle}(y)\\ \nonumber =\int_{\mathcal{I}(\reels^+,\reels^+)} \proba \big(\inf \{s \geq 0\text{ s.t. } B_{s} \geq 1\} \leq y_t \big)dF_{\langle Y\rangle}(y).
\end{eqnarray}
Then, we can calculate that
\begin{eqnarray}
\nonumber && \int_{\mathcal{I}(\reels^+,\reels^+)} \proba \big(\inf \{s \geq 0\text{ s.t. } B_{s} \geq 1\} \leq y_t \big)dF_{\langle Y\rangle}(y)\\ \nonumber & = & \int_{\mathcal{I}(\reels^+,\reels^+)} \proba \big(\inf \{s \geq 0\text{ s.t. } W_{s} \geq 1\} \leq y_t \big)dF_{\langle Y\rangle}(y)\\ \nonumber & = & \int_{\mathcal{I}(\reels^+,\reels^+)} \proba \big(\Tau^W_1  \leq y_t \big)dF_{\langle Y\rangle}(y)\\ \label{eq_proof_theorem_rc3}
& = & \int_{\mathcal{I}(\reels^+,\reels^+)} P_1^W\big(y_t\big) dF_{\langle Y\rangle}(y), 
\end{eqnarray}
where we use the fact that $B$ and $W$ have the same distribution in the first equality, Equation (\ref{TgZdef_rc1}) in the second equality, and Equation (\ref{PgZdef_rc1}) in the third equality. By Equations (\ref{eq_proof_theorem_rc1}), (\ref{eq_proof_theorem_rc2}) and (\ref{eq_proof_theorem_rc3}), we can deduce Equation (\ref{eq_theorem_rc1}).
\end{proof}
\noindent Finally, we give the proof of Corollary \ref{corollary_rc1}.
\begin{proof}[Proof of Corollary \ref{corollary_rc1}]
We have for $t \geq 0$
\begin{eqnarray*}
f_1^Y (t) & = & \frac{dP_1^Y(t)}{dt}\\
& = & \frac{d(\int_{\mathcal{I}(\reels^+,\reels^+)} P_1^W \big(y_t\big) dF_{\langle Y\rangle}(y))}{dt}\\
& = & \int_{\mathcal{I}(\reels^+,\reels^+)} \frac{d(P_1^W \big(y_t\big))}{dt} dF_{\langle Y\rangle}(y)\\
& = & \int_{\mathcal{I}(\reels^+,\reels^+)} y_t ' f_1^W \big(y_t\big) dF_{\langle Y\rangle}(y),
\end{eqnarray*}
where we use Equation (\ref{fZgt_rc1}) in the first equality, Equation (\ref{eq_theorem_rc1}) from Theorem \ref{theorem_rc1} along with Assumption \ref{assDDS_rc1} in the second equality, Tonelli's theorem in the third equality, and the fundamental theorem of calculus along with chain rule and the assumption that the quadratic variation $\langle Y\rangle$ is absolutely continuous on $\reels^+$ in the fourth equality.
\end{proof}

\subsection{Two-sided (ii) case}
\noindent We give now the proof of Proposition \ref{lemmaDDS_rc2}. The main idea to prove it is the scale invariant property of the time-changed Wiener process and thus the scale invariant property of the FPT which adapts to the two-sided 
 (ii) case. 
\begin{proof}[Proof of Proposition \ref{lemmaDDS_rc2}]
We have that for any $t \geq 0$
\begin{eqnarray}
\nonumber \big\{ \Tau_{g,h}^Z =t \big\} 
& = & \big\{ \inf \{s \geq 0 \text{ s.t. } Z_s \geq g \text{ or } Z_s \leq h\} = t \big\}\\ \label{eq_proof_lemmaDDS_rc21} 
& = & \big\{ \inf \{s \geq 0\text{ s.t. } B_{\langle Z\rangle_{s}} \geq g \text{ or } B_{\langle Z\rangle_{s}} \leq h\} = t \big\},
\end{eqnarray}
where we use Equation (\ref{TgZdef_rc2}) in the first equality, and Equation (\ref{proof230110_rc2}) in the second equality. Since $B$ is a $({\mathcal {F}}_{\langle Z\rangle_{t}^{-1}})_{t\geq 0}$-Wiener process, $W$ is an $({\mathcal {F}}_{t})_{t\geq 0}$-Wiener process and the two-sided boundary is constant, we can use a random time change of quadratic variation $\langle Z\rangle_t$ value to obtain that
\begin{eqnarray}
\label{eq_proof_lemmaDDS_rc22} \big\{ \inf \{s \geq 0\text{ s.t. } B_{\langle Z\rangle_{s}} \geq g \text{ or } B_{\langle Z\rangle_{s}} \leq h\} = t \big\} \\ \nonumber = \big\{\inf \{s \geq 0 \text{ s.t. } B_s \geq g \text{ or } B_{s} \leq h \} = \langle Z\rangle_t \big\}.
\end{eqnarray}
Then, we can calculate by Equation (\ref{TgZdef_rc2}) that
\begin{eqnarray}
\label{eq_proof_lemmaDDS_rc23} \big\{\inf \{s \geq 0 \text{ s.t. } B_s \geq g \text{ or } B_{s} \leq h \} = \langle Z\rangle_t \big\} & = & \big\{ \Tau^B_{g,h} = \langle Z\rangle_t \big\}. 
\end{eqnarray}
By Equations (\ref{eq_proof_lemmaDDS_rc21}), (\ref{eq_proof_lemmaDDS_rc22}) and (\ref{eq_proof_lemmaDDS_rc23}), we can deduce Equation (\ref{eq_lemmaDDS_rc2}).
\end{proof}
\noindent In what follows, we give the proof of Theorem \ref{theorem_rc2}. The proof is mainly based on Proposition \ref{lemmaDDS_rc2}. We get $P_{g,h}^Z$ in the next theorem by integrating $\proba (\Tau^{Z}_{g,h} \leq t | u=(g_0,h_0,z))$ with respect to the value of $(g_0,h_0,z) \in \mathcal{S}$.
\begin{proof}[Proof of Theorem \ref{theorem_rc2}]
We have that for any $t \geq 0$
\begin{eqnarray}
\nonumber P_{g,h}^Z(t) & = & \proba (\Tau^Z_{g,h} \leq t)\\ \nonumber & = & \int_{\mathcal{S}} \proba (\Tau^{Z}_{g,h} \leq t | u=(g_0,h_0,z)) dF_{u}(g_0,h_0,z)\\ \nonumber 
& = & \int_{\mathcal{S}} \proba \big(\inf \{s \geq 0 \text{ s.t. } Z_s \geq g \text{ or } Z_s \leq h \} \leq t |u=(g_0,h_0,z)) dF_{u}(g_0,h_0,z)\\ \label{eq_proof_theorem_rc21} 
& = & \int_{\mathcal{S}} \proba \big(\inf \{s \geq 0 \text{ s.t. } B_{\langle Z\rangle_{s}}  \geq g \text{ or } B_{\langle Z\rangle_{s}}  \leq h \} \leq t |u=(g_0,h_0,z)) dF_{u}(g_0,h_0,z),
\end{eqnarray}
where we use Equation (\ref{PgZdef_rc2}) in the first equality, regular conditional probability in the second equality, Equation (\ref{TgZdef_rc2}) in the third equality, and Equation (\ref{proof230110_rc2}) in the fourth equality. Since the stochastic process $Z$ is independent from the two-sided boundary $(g,h)$, we obtain that
\begin{eqnarray}
\label{eq_proof_theorem_rc211} 
&  & \int_{\mathcal{S}} \proba \big(\inf \{s \geq 0 \text{ s.t. } B_{\langle Z\rangle_{s}}  \geq g \text{ or } B_{\langle Z\rangle_{s}}  \leq h \} \leq t |u=(g_0,h_0,z)) dF_{u}(g_0,h_0,z) \\ \nonumber &=& \int_{\mathcal{S}} \proba \big(\inf \{s \geq 0 \text{ s.t. } B_{z_{s}}  \geq g_0 \text{ or } B_{z_{s}}  \leq h_0 \} \leq t) dF_{u}(g_0,h_0,z).
\end{eqnarray}
By Lemma \ref{lemmaDDS_rc2} along with Assumption \ref{assDDS_rc2}, we obtain that
\begin{eqnarray}
\label{eq_proof_theorem_rc22}&&\int_{\mathcal{S}} \proba \big(\inf \{s \geq 0 \text{ s.t. } B_{z_{s}}  \geq g_0 \text{ or } B_{z_{s}}  \leq h_0 \} \leq t) dF_{u}(g_0,h_0,z)\\ \nonumber &=& \int_{\mathcal{S}} \proba \big(\inf \{s \geq 0 \text{ s.t. } B_{s}  \geq g_0 \text{ or } B_{s}  \leq h_0 \} \leq z_t) dF_{u}(g_0,h_0,z).
\end{eqnarray}
Then, we can calculate that
\begin{eqnarray}
\nonumber && \int_{\mathcal{S}} \proba \big(\inf \{s \geq 0 \text{ s.t. } B_{s}  \geq g_0 \text{ or } B_{s}  \leq h_0 \} \leq z_t) dF_{u}(g_0,h_0,z)\\ \nonumber & = & \int_{\mathcal{S}} \proba \big(\inf \{s \geq 0 \text{ s.t. } W_{s}  \geq g_0 \text{ or } W_{s}  \leq h_0 \} \leq z_t) dF_{u}(g_0,h_0,z)\\ \nonumber & = & \int_{\mathcal{S}} \proba \big(\Tau^W_{g_0,h_0}  \leq z_t \big)dF_{u}(g_0,h_0,z)\\ \label{eq_proof_theorem_rc23}
& = & \int_{\mathcal{S}} P_{g_0,h_0}^W\big(z_t\big) dF_{u}(g_0,h_0,z), 
\end{eqnarray}
where we use the fact that $B$ and $W$ have the same distribution in the first equality, Equation (\ref{TgZdef_rc2}) in the second equality, and Equation (\ref{PgZdef_rc2}) in the third equality. By Equations (\ref{eq_proof_theorem_rc21}), (\ref{eq_proof_theorem_rc211}),(\ref{eq_proof_theorem_rc22}) and (\ref{eq_proof_theorem_rc23}), we can deduce Equation (\ref{eq_theorem_rc2}).
\end{proof}
\noindent Finally, we give the proof of Corollary \ref{corollary_rc2}.
\begin{proof}[Proof of Corollary \ref{corollary_rc2}]
We have for $t \geq 0$
\begin{eqnarray*}
f_{g,h}^Z (t) & = & \frac{dP_{g,h}^Z(t)}{dt}\\
& = & \frac{d(\int_{\mathcal{S}} P_{g_0,h_0}^W \big(z_t\big) dF_{u}(g_0,h_0,z))}{dt}\\
& = & \int_{\mathcal{S}} \frac{d(P_{g_0,h_0}^W \big(z_t\big))}{dt} dF_{u}(g_0,h_0,z)\\
& = & \int_{\mathcal{S}} z_t ' f_{g_0,h_0}^W \big(z_t\big) dF_{u}(g_0,h_0,z),
\end{eqnarray*}
where we use Equation (\ref{fZgt_rc2}) in the first equality, Equation (\ref{eq_theorem_rc2}) from Theorem \ref{theorem_rc2} along with Assumption \ref{assDDS_rc2} and the assumption that the stochastic process $Z$ is independent from the two-sided boundary $(g,h)$ in the second equality, Tonelli's theorem in the third equality, and the fundamental theorem of calculus along with chain rule and the assumption that the quadratic variation $\langle Z\rangle$ is absolutely continuous on $\reels^+$ in the fourth equality.
\end{proof}

\section{Proofs of the inverse FPT problem}
\label{section_proofs_application}
In this section, we prove the explicit solution of the inverse FPT problem for the one-sided and two-sided boundary in the case (i) and (ii).
\subsection{One-sided (i) case}

We first give the proof of Lemma \ref{lemma_inv_nrc1}. The proof relies on Lemma \ref{lemma_nrc1}.
\begin{proof}[Proof of Lemma \ref{lemma_inv_nrc1}]
Using Equation (\ref{eqP_lemma_nrc1}) from Lemma \ref{lemma_nrc1}, Equation (\ref{def_standardgaussian}) along with Equation (\ref{def_erf}), we can express the relation between the cdf of the standard normal and the error function as
\begin{equation}
 \label{Phirel}   \Phi(x)=\frac{1}{2}\Big(1+\text{erf}(\frac{x}{\sqrt{2}})\Big).
\end{equation}
We can rewrite Equation (\ref{eqP_lemma_nrc1}) as
\begin{align*}
P_g^W(t)&=1-\frac{1}{2}\Big(1+\text{erf}(\frac{g}{\sqrt{2t}})\Big)+\frac{1}{2}\Big(1+\text{erf}(\frac{-g}{\sqrt{2t}})\Big)\\
&=1-\frac{1}{2}\Big(1+\text{erf}(\frac{g}{\sqrt{2t}})\Big)+\frac{1}{2}\Big(1-\text{erf}(\frac{g}{\sqrt{2t}})\Big)\\
&=1-\text{erf}(\frac{g}{\sqrt{2t}}),
\end{align*}
We note that $P_g^W : \reels^+ \rightarrow [0,1)$ is strictly increasing since $f_g^W(t) > 0$ for any $t > 0$ by Equation (\ref{eqf_lemma_nrc1}). Thus, there exists an invert $(P_g^W)^{-1} : [0,1) \rightarrow \reels^+$ which is strictly increasing. First, note that as $P_g^{W}(0)=0$, this implies that $(P_g^{W})^{-1}(0)=0$. 
Using Equation (\ref{eqP_lemma_nrc1}), some algebraic manipulation leads to Equation (\ref{PgWinvt}).
Finally, applying Equation (\ref{eqf_lemma_nrc1}) yields the form of $f_g^W((P_g^{W^{-1}}(t))$, i.e., Equation (\ref{fP1}).
\end{proof}
\noindent We then give a lemma whose proof relies on Lemma \ref{lemma_nrc1} and Lemma \ref{lemma_inv_nrc1}. For $A \subset \reels^+$ and $B \subset \reels^+$, we denote the space $\mathcal{C}_1$ of functions $k : A \rightarrow B$ with derivatives which are continuous as $\mathcal{C}_1(A, B)$.
\begin{lemma}
\label{lemma_PgWinvertC1_nrc1}
We have
\begin{eqnarray}
\label{fPgWC1}
f_g^W \in \mathcal{C}_1(\reels^+, \reels^+) \text{ and } P_g^W \in \mathcal{C}_1(\reels^+, [0,1)),\\
\label{PgWinvC1}
(P_g^W)^{-1} \in \mathcal{C}_1([0,1),\reels^+). 
\end{eqnarray}
\end{lemma}
\begin{proof}[Proof of Lemma \ref{lemma_PgWinvertC1_nrc1}] By Equations (\ref{eqP_lemma_nrc1}) and (\ref{eqf_lemma_nrc1}) in Lemma \ref{lemma_nrc1}, we obtain Equation (\ref{fPgWC1}). By Equation (\ref{PgWinvt}) in Lemma \ref{lemma_inv_nrc1}, we obtain Equation (\ref{PgWinvC1}).
\end{proof}
\noindent The next proposition shows that Assumption \ref{assmain_F_nrc1} implies that $Z^{F}$ satisfies Assumption \ref{assDDS_nrc1}.
\begin{proposition}
\label{prop_ass_F_nrc1}
Under Assumption \ref{assmain_F_nrc1}, we have that $Z^{F}$ satisfies Assumption \ref{assDDS_nrc1}.
\end{proposition}

\begin{proof}[Proof of Proposition \ref{prop_ass_F_nrc1}] By Definition \ref{def_explicitsolution_F_nrc1}, $Z^F$ is defined as a continuous local martingale with quadratic variation $\langle Z^F\rangle_{t} = v_{F}(t) \text{ for } t \geq 0$, which can be expressed by Equation (\ref{expsoldefsimp_nrc1}) as 
\begin{eqnarray*}
v_{F}(t) = & \frac{g^2}{2\operatorname{erfinv}(1-F(t))^2} \mathbf{1}_{\{0 < F(t) < 1\}} & \text{ for } t \geq 0.
\end{eqnarray*}
By definition we have that $\operatorname{erfinv}(z) \rightarrow 0$ as $z\rightarrow 0$, and by Definition \ref{Fdef_nrc1} we have that $\underset{t \rightarrow \infty}{\lim} F(t) =1$. Thus, we can deduce by Assumption \ref{assmain_F_nrc1} that $\underset{t \rightarrow \infty}{\lim} v_F(t) =\infty$. This implies that $\langle Z^F \rangle_{\infty} = \infty$ and thus that $Z^{F}$ satisfies Assumption \ref{assDDS_nrc1}.
\end{proof}
\noindent The next proposition states that if a nondecreasing function satisfies Assumption \ref{assDDS_nrc1}, then it is a solution if and only if it is an explicit solution. The proof is based on substituting the left-hand side of Equation (\ref{PgZF_nrc1}) with Equations (\ref{eq_theorem_nrc1}) and (\ref{v_F_def_nrc1}) and then inverting on both sides of the equation to derive the explicit solution.
\begin{proposition}
\label{propiii_F_nrc1}
If we assume that $v_F$ satisfies Assumption \ref{assDDS_nrc1}, then, (i) $v_F$ is a solution in the sense of Definition \ref{def_solution_F_nrc1} $\iff$ (ii) $v_F$ is an explicit solution in the sense of Definition \ref{def_explicitsolution_F_nrc1}.
\end{proposition}
\begin{proof}[Proof of Proposition \ref{propiii_F_nrc1}] Proof of $(i) \implies (ii)$. We assume that $v_F$ is a solution in the sense of Definition \ref{def_solution_F_nrc1}. Given that $Z^F$ satisfies Assumption \ref{assDDS_nrc1}, we can substitute the left-hand side of Equation (\ref{PgZF_nrc1}) with Equation (\ref{eq_theorem_nrc1}) to deduce 
\begin{eqnarray}
\label{proofiii_F_eq0_nrc1}
P_g^W \Big(\langle Z^F\rangle_{t} \Big) = F(t) \text{ for } t \geq 0. 
\end{eqnarray}
Using Equation (\ref{v_F_def_nrc1}), Equation (\ref{proofiii_F_eq0_nrc1}) can be reexpressed as 
\begin{eqnarray}
\label{proofiii_F_eq1_nrc1}
P_g^W \Big(v_F(t) \Big) = F(t) \text{ for } t \geq 0. 
\end{eqnarray}
By Lemma \ref{lemma_inv_nrc1}, there exists an invert $(P_g^W)^{-1} : [0,1) \rightarrow \reels^+$.
Applying $(P_g^W)^{-1}$ on both sides of Equation (\ref{proofiii_F_eq1_nrc1}), Equation (\ref{proofiii_F_eq1_nrc1}) can be rewritten as
Equation (\ref{def_expsol_F_nrc1}).

\noindent Proof of $(ii) \implies (i)$. We assume that $v_F$ is an explicit solution in the sense of Definition \ref{def_explicitsolution_F_nrc1}. We have
\begin{eqnarray*}
P_g^{Z^F} (t)  & = & P_g^W \Big(\langle Z^F\rangle_{t} \Big) \\
 & = & P_g^W (v_F(t))\\
& = & P_g^W ((P_g^W)^{-1}(F(t)) \mathbf{1}_{\{0 < F(t) < 1\}})\\
& = & F(t),
\end{eqnarray*}
where we use Equation (\ref{eq_theorem_nrc1}) along with the assumption that $v_F$ satisfies Assumption \ref{assDDS_nrc1} in the first equality, Equation (\ref{v_F_def_nrc1}) in the second equality, Equation (\ref{def_expsol_F_nrc1}) in the third equality, and algebraic manipulation in the fourth equality.
\end{proof}
\noindent The proof of Theorem \ref{theorem_F_nrc1} is a direct application of Proposition \ref{prop_ass_F_nrc1} and Proposition \ref{propiii_F_nrc1}.
\begin{proof}[Proof of Theorem \ref{theorem_F_nrc1}] To obtain (a), we apply Proposition \ref{prop_ass_F_nrc1} along with Assumption \ref{assmain_F_nrc1}. Then, an application of Proposition \ref{propiii_F_nrc1} along with (a) yields (b). 
\end{proof}

\noindent The next proposition shows that Assumption \ref{assmain_f_nrc1} implies that $Z^{f}$ satisfies Assumption \ref{assDDS_nrc1}. The proof is mainly based on topological argument in $\reels^+$ and the use of Assumption \ref{assmain_f_nrc1}. 
\begin{proposition}
\label{prop_ass_f_nrc1}
Under Assumption \ref{assmain_f_nrc1}, we have that $Z^{f}$ satisfies Assumption \ref{assDDS_nrc1}.
\end{proposition}

\begin{proof}[Proof of Proposition \ref{prop_ass_f_nrc1}] To prove that $Z^{f}$ satisfies Assumption \ref{assDDS_nrc1}, we first show that $\sigma_{f}^2 \in L_{1,\mathrm{loc}}(\reels^+)$, i.e., we have to show by definition that $\forall\, K \subset \reels^+,\, K \text{ compact}$, we have 
\begin{eqnarray}
\label{proofassA}
\int_K \sigma_{t,f}^2 \,\mathrm{d}t <+\infty.
\end{eqnarray}
There is no loss of generality assuming that $K$ has a closed interval form, e.g., $K=[K_0,K_1]$ where $0 \leq K_0 < K_1$, since if not we can break $K$ into a finite number of nonoverlapping closed intervals by the Bolzano-Weierstrass theorem and prove Equation (\ref{proofassA}) for each interval.
We first consider the case where
\begin{eqnarray}
\label{proofcompactcase1}
0 \leq K_F^0 < K_0 < K_1.
\end{eqnarray}
Given the form of the explicit solution (\ref{def_expsol_f_nrc1}), Equation (\ref{proofassA}) can be reexpressed as
\begin{eqnarray}
\label{proof0225}
\int_K \frac{f(t)}{f_g^W((P_g^W)^{-1}(F(t)))} \,\mathrm{d}t <+\infty.
\end{eqnarray}
We first show that the denominator in the integral of Equation (\ref{proof0225}) is uniformly bounded away from $0$ on $K$. By Definition \ref{Fdef_nrc1}, $F$ is a cdf and thus a nondecreasing function. We can deduce that \begin{eqnarray}
\label{proof1_f_nrc1}
F(K_0) \leq F(t) \leq F(K_1) \text{ for } t \in K. 
\end{eqnarray}
We also obtain by definition of $K_F^0$ in Equation (\ref{defKF0}), definition of $K_F^1$ in Equation (\ref{defKF1}) and the assumption that $K_F^1$ is not finite from Assumption \ref{assmain_f_nrc1} that 
\begin{eqnarray}
\label{proof2_f_nrc1}
0 < F(\widetilde{K}) < 1 \text{ for } \widetilde{K} \in \reels \text{ such that } K_F^0 < \widetilde{K}.
\end{eqnarray}
Combining Equations (\ref{proof1_f_nrc1}) and (\ref{proof2_f_nrc1}), we can deduce that 
\begin{eqnarray}
\label{Fineq}
0< F(K_0) \leq F(t) \leq F(K_1) < 1, \text{ for }t \in K.
\end{eqnarray}
By Lemma \ref{lemma_inv_nrc1}, we have that $(P_g^W)^{-1}$ is strictly increasing. Thus, applying $(P_g^W)^{-1}$ to each term of Inequality (\ref{Fineq}) yields
\begin{eqnarray}
\label{Fineq0}
0< (P_g^W)^{-1}(F(K_0)) \leq (P_g^W)^{-1}(F(t))) \leq (P_g^W)^{-1}(F(K_1)), \text{ for }t \in K.
\end{eqnarray}
We have that $(P_g^W)^{-1}(F(t)))$ takes its values in the closed interval $$[(P_g^W)^{-1}(F(K_0)), (P_g^W)^{-1}(F(K_1))]$$ of $\reels^+$ which is connected and compact by the Bolzano-Weierstrass theorem. Besides, it is known from topological properties that the image of a compact and connected set of $\reels^+$ by a continuous function from $\reels^+$ to $\reels^+$ is a compact and connected set of $\reels^+$. Since $f_g^W$ is continuous by Equation (\ref{eqf_lemma_nrc1}), we can deduce that $f_g^W ((P_g^W)^{-1}(F(t)))$ for any $t \in K$ is included into a compact and connected space of $\reels^+$, e.g., a closed interval of $\reels^+$. From Equation (\ref{eqf_lemma_nrc1}), we get that there exists $C > 0$ such that
\begin{eqnarray}
\label{Fineq2}
C \leq f_g^W ((P_g^W)^{-1}(F(t))) \text{ for any }t \in K.
\end{eqnarray}
This implies that the denominator in the integral of Equation (\ref{proof0225}) is uniformly bounded away from $0$ on $K$. Given that $f$ is a pdf, we obtain that
\begin{eqnarray*}
\int_K f(t)\,\mathrm{d}t <+\infty.
\end{eqnarray*}
Thus, Equation (\ref{proof0225}) holds. We now consider the general case when K is not necessarily of the form (\ref{proofcompactcase1}). We consider the case when $K_0 \leq K_F^0 < K_1$. If we introduce the notation  $\widetilde{K}_F^0 = K_F^0 + \eta_F^0$, then we can decompose $[K_0,K_1]$ as
$$[K_0,K_1] = [K_0, K_F^0] \cup [K_F^0, \widetilde{K}_F^0] \cup [\widetilde{K}_F^0, K_1].$$
We deduce that
\begin{eqnarray*}
\int_K  \sigma_{t,f}^2 \,\mathrm{d}t & = & \int_{[K_0, K_F^0]}  \sigma_{t,f}^2 \,\mathrm{d}t + \int_{[K_F^0, \widetilde{K}_F^0]}  \sigma_{t,f}^2 \,\mathrm{d}t + \int_{[\widetilde{K}_F^0, \widetilde{K}_F^1]}  \sigma_{t,f}^2 \,\mathrm{d}t \\
& = & \int_{[K_F^0, \widetilde{K}_F^0]}  \sigma_{t,f}^2 \,\mathrm{d}t + \int_{[\widetilde{K}_F^0, \widetilde{K}_F^1]}  \sigma_{t,f}^2 \,\mathrm{d}t\\
& \leq & C + \int_{[\widetilde{K}_F^0, \widetilde{K}_F^1]}  \sigma_{t,f}^2 \,\mathrm{d}t\\
& < & + \infty,
\end{eqnarray*}
where the second equality is due to the fact that the variance function is null on $[K_0, K_F^0]$ by Equation (\ref{def_expsol_f_nrc1}),  the first inequality with $C >0$ follows by Expression (\ref{assumptionvol1}) from Assumption \ref{assmain_f_nrc1}, and the second inequality is due to Equation (\ref{proof0225}). Finally, we have that the variance function is null on by Equation (\ref{def_expsol_f_nrc1}) in the case when $K_F^0 \leq K_0$. We have thus shown Expression  (\ref{proofassA}). Thus, we can deduce that $Z^f$ is a local martingale with nonrandom quadratic variation 
\begin{eqnarray}
\label{proof_nrc1}
\langle Z^f\rangle_{t}  &=& \int_0^t \sigma_{u,f}^2 du
\end{eqnarray} 
by Theorem I.4.40 (p. 48) from \cite{JacodLimit2003} along with Expression (\ref{proofassA}). Finally, we show that $\langle Z^f\rangle_{t} \rightarrow \infty$ as $t \rightarrow \infty$. We can calculate that 
\begin{eqnarray}
\nonumber \langle Z^f\rangle_{t}  &=& \langle Z^F\rangle_{t}\\ \nonumber &=& v_{F}(t)\\ \label{proof2_nrc1}
&=& \frac{g^2}{2\operatorname{erfinv}(1-F(t))^2} \mathbf{1}_{\{0 < F(t) < 1\}},
\end{eqnarray}
where we use the fact that $Z^f=Z^F$ in the first equality,  Equation (\ref{v_F_def_nrc1}) from Definition \ref{def_explicitsolution_F_nrc1} in the second equality, and Equation (\ref{expsoldefsimp_nrc1}) in the last equality. By definition we have that $\operatorname{erfinv}(z) \rightarrow 0$ as $z\rightarrow 0$, and by Definition \ref{Fdef_nrc1} we have that $\underset{t \rightarrow \infty}{\lim} F(t) =1$. Thus, we can deduce by the assumption that $K_F^1$ is finite from Assumption \ref{assmain_f_nrc1} that 
\begin{eqnarray}
\label{proof3_nrc1}
\frac{g^2}{2\operatorname{erfinv}(1-F(t))^2} \mathbf{1}_{\{0 < F(t) < 1\}} \rightarrow 0
\end{eqnarray}
as $t \rightarrow \infty$. We can deduce by Equations (\ref{proof_nrc1}), (\ref{proof2_nrc1}) and (\ref{proof3_nrc1})
that $\langle Z^f\rangle_{t} \rightarrow \infty$ as $t \rightarrow \infty$. This implies that $Z^{f}$ satisfies Assumption \ref{assDDS_nrc1}.
 
\end{proof}
\noindent The next proposition states that if $Z^f$ satisfies Assumption \ref{assDDS_nrc1}, then, the variance function is a solution if and only if it is an explicit solution. The proof is based on substituting the left-hand side of Equation (\ref{PgZf_nrc1}) with Equations (\ref{eq_theorem_nrc1}) from Theorem \ref{theorem_nrc1} and (\ref{sigma_f_def_nrc1}) and then differentiating and inverting on both sides of the equation to derive the explicit solution.
\begin{proposition}
\label{propiii_f_nrc1}
If we assume that $Z^f$ satisfies Assumption \ref{assDDS_nrc1}, then, (i) $\sigma_f^2$ is a solution in the sense of Definition \ref{def_solution_f_nrc1} $\iff$ (ii) $\sigma_f^2$ is an explicit solution in the sense of Definition \ref{def_explicitsolution_f_nrc1}.
\end{proposition}
\begin{proof}[Proof of Proposition \ref{propiii_f_nrc1}] Proof of $(i) \implies (ii)$. We assume that $\sigma_f^2$ is a solution in the sense of Definition \ref{def_solution_f_nrc1}. Given that $Z^f$ satisfies Assumption \ref{assDDS_nrc1}, we can substitute the left-hand side of Equation (\ref{PgZf_nrc1}) with Equation (\ref{eq_theorem_nrc1}) to deduce 
\begin{eqnarray}
\label{proofiii_f_eq0_nrc1}
P_g^W \Big(\langle Z^f\rangle_{t} \Big) = F(t) \text{ for } t \geq 0. 
\end{eqnarray}
Using Equation (\ref{sigma_f_def_nrc1}), Equation (\ref{proofiii_f_eq0_nrc1}) can be reexpressed as 
\begin{eqnarray}
\label{proofiii_f_eq1_nrc1}
P_g^W \Big(\int_0^t \sigma_{s,f}^2 ds \Big) = F(t) \text{ for } t \geq 0. 
\end{eqnarray}
By Lemma \ref{lemma_inv_nrc1}, there exists an invert $(P_g^W)^{-1} : [0,1) \rightarrow \reels^+$.
Applying $(P_g^W)^{-1}$ on both sides of Equation (\ref{proofiii_f_eq1_nrc1}), Equation (\ref{proofiii_f_eq1_nrc1}) can be rewritten as
\begin{eqnarray}
\label{proofiii_f_eq2_nrc1}
\int_0^t \sigma_{s,f}^2 ds  = & (P_g^W)^{-1}(F(t)) \mathbf{1}_{\{0 < F(t) < 1\}} & \text{ for } t \geq 0.
\end{eqnarray}
The left-hand side of Equation (\ref{proofiii_f_eq2_nrc1}) and $F$ have a derivative almost everywhere for any $t \geq 0$ by absolute continuity properties and since $F$ is absolutely continuous. $(P_g^W)^{-1}$ is differentiable on $[0,1)$ by Lemma \ref{lemma_PgWinvertC1_nrc1}. Thus, we can differentiate (\ref{proofiii_f_eq2_nrc1}) almost everywhere on both sides, by using the chain rule on the right-hand side. We obtain
\begin{eqnarray}
\label{sigmatff}
\sigma_{t,f}^2 = & f(t)((P_g^W)^{-1})'(F(t)) \mathbf{1}_{\{0 < F(t) < 1\}}& \text{almost everywhere for } t \geq 0.
\end{eqnarray}
Applying the inverse function theorem, Equation (\ref{sigmatff}) can be reexpressed as
\begin{eqnarray*}
\sigma_{t,f}^2 = &\frac{f(t)}{(P_g^W)'((P_g^W)^{-1}(F(t)))} \mathbf{1}_{\{0 < F(t) < 1\}} & \text{ almost everywhere for } t \geq 0,
\end{eqnarray*}
or equivalently of the form (\ref{def_expsol_f_nrc1}) as $(P_g^W)'(t)=f_g(t)$ almost everywhere for any $t \geq 0$. Thus, we have shown that $\sigma_f^2$ is an explicit solution in the sense of Definition \ref{def_explicitsolution_f_nrc1}.

\noindent Proof of $(ii) \implies (i)$. We assume that $\sigma_f^2$ is an explicit solution in the sense of Definition \ref{def_explicitsolution_f_nrc1}. We have almost everywhere for $t \geq 0$
\begin{eqnarray*}
P_g^{Z^f} (t)  & = & P_g^W (\int_0^t \sigma_{s,f}^2 ds)\\
 & = & P_g^W (\int_0^t \frac{f(s)}{f_g^W((P_g^W)^{-1}(F(s)))} \mathbf{1}_{\{0 < F(t) < 1\}} ds)\\ 
& = & P_g^W (\int_0^t f(s)((P_g^W)^{-1})'(F(s)) \mathbf{1}_{\{0 < F(t) < 1\}} ds)\\
& = & P_g^W ((P_g^W)^{-1})(F(t))\\
& = & F(t),
\end{eqnarray*}
where we use Equation (\ref{eq_theorem_nrc1}) along with the assumption that $Z^f$ satisfies Assumption \ref{assDDS_nrc1} in the first equality, Equation (\ref{def_expsol_f_nrc1}) in the second equality, the inverse function theorem in the third equality, integration in the fourth equality and algebraic manipulation in the fifth equality.
We have thus shown that $\sigma_t^2$ satisfies Equation (\ref{tildesigma_f_def_nrc1}), and thus that $\sigma_f^2$ is a solution in the sense of Definition \ref{def_solution_f_nrc1}.
\end{proof}
\noindent The proof of Theorem \ref{theorem_f_nrc1} is a direct application of Proposition \ref{prop_ass_f_nrc1} and Proposition \ref{propiii_f_nrc1}.
\begin{proof}[Proof of Theorem \ref{theorem_f_nrc1}] To obtain (a), we apply Proposition \ref{prop_ass_f_nrc1} along with Assumption \ref{assmain_f_nrc1}. Then, an application of Proposition \ref{propiii_f_nrc1} along with (a) yields (b). 
\end{proof}

\subsection{Two-sided (i) case}
We first give the proof of Lemma \ref{lemma_inv_nrc2}. The proof relies on Lemma \ref{lemma_nrc2}.
\begin{proof}[Proof of Lemma \ref{lemma_inv_nrc2}]
Using Equation (\ref{eqP_lemma_nrc2}) from Lemma \ref{lemma_nrc2}, we note that $P_{g,h}^W : \reels^+ \rightarrow [0,1)$ is strictly increasing since $f_{g,h}^W(t) > 0$ for any $t > 0$ by Equation (\ref{eqf_lemma_nrc2}). Thus, there exists an invert $(P_{g,h}^W)^{-1} : [0,1) \rightarrow \reels^+$ which is strictly increasing. First, note that as $P_{g,h}^{W}(0)=0$, this implies that $(P_{g,h}^{W})^{-1}(0)=0$. 
Using Equation (\ref{eqP_lemma_nrc2}), some algebraic manipulation leads to $(P_{g,h}^{W})^{-1}(1)=\infty$.
\end{proof}
\noindent We then give a lemma whose proof relies on Lemma \ref{lemma_nrc2}.
\begin{lemma}
\label{lemma_PgWinvertC1_nrc2}
We have
\begin{eqnarray}
\label{fPgWC1_nrc2}
f_{g,h}^W \in \mathcal{C}_1(\reels^+, \reels^+) \text{ and } P_{g,h}^W \in \mathcal{C}_1(\reels^+, [0,1)),\\
\label{PgWinvC1_nrc2}
(P_{g,h}^W)^{-1} \in \mathcal{C}_1([0,1),\reels^+). 
\end{eqnarray}
\end{lemma}
\begin{proof}[Proof of Lemma \ref{lemma_PgWinvertC1_nrc2}] By Equations (\ref{eqP_lemma_nrc2}) and (\ref{eqf_lemma_nrc2}) in Lemma \ref{lemma_nrc2}, we obtain Equations (\ref{fPgWC1_nrc2}) and (\ref{PgWinvC1_nrc2}).
\end{proof}
\noindent The next proposition shows that Assumption \ref{assmain_F_nrc1} implies that $Z^{F}$ satisfies Assumption \ref{assDDS_nrc1}.
\begin{proposition}
\label{prop_ass_F_nrc2}
Under Assumption \ref{assmain_F_nrc1}, we have that $Z^{F}$ satisfies Assumption \ref{assDDS_nrc1}.
\end{proposition}

\begin{proof}[Proof of Proposition \ref{prop_ass_F_nrc2}] By Definition \ref{def_explicitsolution_F_nrc2}, $Z^F$ is defined as a continuous local martingale with quadratic variation $\langle Z^F\rangle_{t} = v_{F}(t) \text{ for } t \geq 0$, which can be expressed  as 
\begin{eqnarray*}
v_{F}(t) = & (P_{g,h}^W)^{-1}(F(t)) \mathbf{1}_{\{0 < F(t) < 1\}}& \text{ for } t \geq 0.
\end{eqnarray*}
By Lemma \ref{lemma_inv_nrc2} we have that $(P_{g,h}^{W})^{-1}(1)=\infty$, and by Definition \ref{Fdef_nrc1} we have that $\underset{t \rightarrow \infty}{\lim} F(t) =1$. Thus, we can deduce by Assumption \ref{assmain_F_nrc1} that $\underset{t \rightarrow \infty}{\lim} v_F(t) =\infty$. This implies that $\langle Z^F \rangle_{\infty} = \infty$ and thus that $Z^{F}$ satisfies Assumption \ref{assDDS_nrc1}.
\end{proof}
\noindent The next proposition states that if a nondecreasing function satisfies Assumption \ref{assDDS_nrc1}, then it is a solution if and only if it is an explicit solution. The proof is based on substituting the left-hand side of Equation (\ref{PgZF_nrc2}) with Equations (\ref{eq_theorem_nrc2}) and (\ref{v_F_def_nrc2}) and then inverting on both sides of the equation to derive the explicit solution.
\begin{proposition}
\label{propiii_F_nrc2}
If we assume that $v_F$ satisfies Assumption \ref{assDDS_nrc1}, then, (i) $v_F$ is a solution in the sense of Definition \ref{def_solution_F_nrc2} $\iff$ (ii) $v_F$ is an explicit solution in the sense of Definition \ref{def_explicitsolution_F_nrc2}.
\end{proposition}
\begin{proof}[Proof of Proposition \ref{propiii_F_nrc2}] Proof of $(i) \implies (ii)$. We assume that $v_F$ is a solution in the sense of Definition \ref{def_solution_F_nrc2}. Given that $Z^F$ satisfies Assumption \ref{assDDS_nrc1}, we can substitute the left-hand side of Equation (\ref{PgZF_nrc2}) with Equation (\ref{eq_theorem_nrc2}) to deduce 
\begin{eqnarray}
\label{proofiii_F_eq0_nrc2}
P_{g,h}^W \Big(\langle Z^F\rangle_{t} \Big) = F(t) \text{ for } t \geq 0. 
\end{eqnarray}
Using Equation (\ref{v_F_def_nrc2}), Equation (\ref{proofiii_F_eq0_nrc2}) can be reexpressed as 
\begin{eqnarray}
\label{proofiii_F_eq1_nrc2}
P_{g,h}^W \Big(v_F(t) \Big) = F(t) \text{ for } t \geq 0. 
\end{eqnarray}
By Lemma \ref{lemma_inv_nrc2}, there exists an invert $(P_{g,h}^W)^{-1} : [0,1) \rightarrow \reels^+$.
Applying $(P_{g,h}^W)^{-1}$ on both sides of Equation (\ref{proofiii_F_eq1_nrc2}), Equation (\ref{proofiii_F_eq1_nrc2}) can be rewritten as
Equation (\ref{def_expsol_F_nrc2}).

\noindent Proof of $(ii) \implies (i)$. We assume that $v_F$ is an explicit solution in the sense of Definition \ref{def_explicitsolution_F_nrc2}. We have
\begin{eqnarray*}
P_{g,h}^{Z^F} (t)  & = & P_{g,h}^W \Big(\langle Z^F\rangle_{t} \Big) \\
 & = & P_{g,h}^W (v_F(t))\\
& = & P_{g,h}^W ((P_g^W)^{-1}(F(t)) \mathbf{1}_{\{0 < F(t) < 1\}})\\
& = & F(t),
\end{eqnarray*}
where we use Equation (\ref{eq_theorem_nrc2}) along with the assumption that $v_F$ satisfies Assumption \ref{assDDS_nrc1} in the first equality, Equation (\ref{v_F_def_nrc2}) in the second equality, Equation (\ref{def_expsol_F_nrc2}) in the third equality, and algebraic manipulation in the fourth equality.
\end{proof}
\noindent The proof of Theorem \ref{theorem_F_nrc2} is a direct application of Proposition \ref{prop_ass_F_nrc2} and Proposition \ref{propiii_F_nrc2}.
\begin{proof}[Proof of Theorem \ref{theorem_F_nrc2}] To obtain (a), we apply Proposition \ref{prop_ass_F_nrc2} along with Assumption \ref{assmain_F_nrc1}. Then, an application of Proposition \ref{propiii_F_nrc2} along with (a) yields (b). 
\end{proof}

\noindent The next proposition shows that Assumption \ref{assmain_f_nrc2} implies that $Z^{f}$ satisfies Assumption \ref{assDDS_nrc1}. The proof is mainly based on topological argument in $\reels^+$ and the use of Assumption \ref{assmain_f_nrc2}. 
\begin{proposition}
\label{prop_ass_f_nrc2}
Under Assumption \ref{assmain_f_nrc2}, we have that $Z^{f}$ satisfies Assumption \ref{assDDS_nrc1}.
\end{proposition}

\begin{proof}[Proof of Proposition \ref{prop_ass_f_nrc2}] To prove that $Z^{f}$ satisfies Assumption \ref{assDDS_nrc1}, we first show that $\sigma_{f}^2 \in L_{1,\mathrm{loc}}(\reels^+)$, i.e., we have to show by definition that $\forall\, K \subset \reels^+,\, K \text{ compact}$, we have 
\begin{eqnarray}
\label{proofassA_nrc2}
\int_K \sigma_{t,f}^2 \,\mathrm{d}t <+\infty.
\end{eqnarray}
There is no loss of generality assuming that $K$ has a closed interval form, e.g., $K=[K_0,K_1]$ where $0 \leq K_0 < K_1$, since if not we can break $K$ into a finite number of nonoverlapping closed intervals by the Bolzano-Weierstrass theorem and prove Equation (\ref{proofassA}) for each interval.
We first consider the case where
\begin{eqnarray}
\label{proofcompactcase1_nrc2}
0 \leq K_F^0 < K_0 < K_1.
\end{eqnarray}
Given the form of the explicit solution (\ref{def_expsol_f_nrc2}), Equation (\ref{proofassA_nrc2}) can be reexpressed as
\begin{eqnarray}
\label{proof0225_nrc2}
\int_K \frac{f(t)}{f_{g,h}^W((P_{g,h}^W)^{-1}(F(t)))} \,\mathrm{d}t <+\infty.
\end{eqnarray}
We first show that the denominator in the integral of Equation (\ref{proof0225_nrc2}) is uniformly bounded away from $0$ on $K$. By Definition \ref{Fdef_nrc1}, $F$ is a cdf and thus a nondecreasing function. We can deduce that \begin{eqnarray}
\label{proof1_f_nrc2}
F(K_0) \leq F(t) \leq F(K_1) \text{ for } t \in K. 
\end{eqnarray}
We also obtain by definition of $K_F^0$ in Equation (\ref{defKF0}), definition of $K_F^1$ in Equation (\ref{defKF1}) and the assumption that $K_F^1$ is not finite from Assumption \ref{assmain_f_nrc2} that 
\begin{eqnarray}
\label{proof2_f_nrc2}
0 < F(\widetilde{K}) < 1 \text{ for } \widetilde{K} \in \reels \text{ such that } K_F^0 < \widetilde{K}.
\end{eqnarray}
Combining Equations (\ref{proof1_f_nrc2}) and (\ref{proof2_f_nrc2}), we can deduce that 
\begin{eqnarray}
\label{Fineq_nrc2}
0< F(K_0) \leq F(t) \leq F(K_1) < 1, \text{ for }t \in K.
\end{eqnarray}
By Lemma \ref{lemma_inv_nrc2}, we have that $(P_{g,h}^W)^{-1}$ is strictly increasing. Thus, applying $(P_{g,h}^W)^{-1}$ to each term of Inequality (\ref{Fineq_nrc2}) yields
\begin{eqnarray}
\label{Fineq0_nrc2}
0< (P_{g,h}^W)^{-1}(F(K_0)) \leq (P_{g,h}^W)^{-1}(F(t))) \leq (P_{g,h}^W)^{-1}(F(K_1)), \text{ for }t \in K.
\end{eqnarray}
We have that $(P_{g,h}^W)^{-1}(F(t)))$ takes its values in the closed interval $$[(P_{g,h}^W)^{-1}(F(K_0)), (P_{g,h}^W)^{-1}(F(K_1))]$$ of $\reels^+$ which is connected and compact by the Bolzano-Weierstrass theorem. Besides, it is known from topological properties that the image of a compact and connected set of $\reels^+$ by a continuous function from $\reels^+$ to $\reels^+$ is a compact and connected set of $\reels^+$. Since $f_{g,h}^W$ is continuous by Equation (\ref{eqf_lemma_nrc2}), we can deduce that $f_{g,h}^W ((P_{g,h}^W)^{-1}(F(t)))$ for any $t \in K$ is included into a compact and connected space of $\reels^+$, e.g., a closed interval of $\reels^+$. From Equation (\ref{eqf_lemma_nrc2}), we get that there exists $C > 0$ such that
\begin{eqnarray}
\label{Fineq2_nrc2}
C \leq f_{g,h}^W ((P_{g,h}^W)^{-1}(F(t))) \text{ for any }t \in K.
\end{eqnarray}
This implies that the denominator in the integral of Equation (\ref{proof0225_nrc2}) is uniformly bounded away from $0$ on $K$. Given that $f$ is a pdf, we obtain that
\begin{eqnarray*}
\int_K f(t)\,\mathrm{d}t <+\infty.
\end{eqnarray*}
Thus, Equation (\ref{proof0225_nrc2}) holds. We now consider the general case when K is not necessarily of the form (\ref{proofcompactcase1_nrc2}). We consider the case when $K_0 \leq K_F^0 < K_1$. If we introduce the notation  $\widetilde{K}_F^0 = K_F^0 + \eta_F^0$, then we can decompose $[K_0,K_1]$ as
$$[K_0,K_1] = [K_0, K_F^0] \cup [K_F^0, \widetilde{K}_F^0] \cup [\widetilde{K}_F^0, K_1].$$
We deduce that
\begin{eqnarray*}
\int_K  \sigma_{t,f}^2 \,\mathrm{d}t & = & \int_{[K_0, K_F^0]}  \sigma_{t,f}^2 \,\mathrm{d}t + \int_{[K_F^0, \widetilde{K}_F^0]}  \sigma_{t,f}^2 \,\mathrm{d}t + \int_{[\widetilde{K}_F^0, \widetilde{K}_F^1]}  \sigma_{t,f}^2 \,\mathrm{d}t \\
& = & \int_{[K_F^0, \widetilde{K}_F^0]}  \sigma_{t,f}^2 \,\mathrm{d}t + \int_{[\widetilde{K}_F^0, \widetilde{K}_F^1]}  \sigma_{t,f}^2 \,\mathrm{d}t\\
& \leq & C + \int_{[\widetilde{K}_F^0, \widetilde{K}_F^1]}  \sigma_{t,f}^2 \,\mathrm{d}t\\
& < & + \infty,
\end{eqnarray*}
where the second equality is due to the fact that the variance function is null on $[K_0, K_F^0]$ by Equation (\ref{def_expsol_f_nrc2}),  the first inequality with $C>0$ follows by Expression (\ref{assumptionvol1_nrc2}) from Assumption \ref{assmain_f_nrc2}, and the second inequality is due to Equation (\ref{proof0225_nrc2}). Finally, we have that the variance function is null on by Equation (\ref{def_expsol_f_nrc2}) in the case when $K_F^0 \leq K_0$. We have thus shown Expression  (\ref{proofassA_nrc2}). Thus, we can deduce that $Z^f$ is a local martingale with nonrandom quadratic variation 
\begin{eqnarray}
\label{proof_nrc2}
\langle Z^f\rangle_{t}  &=& \int_0^t \sigma_{u,f}^2 du
\end{eqnarray} 
by Theorem I.4.40 (p. 48) from \cite{JacodLimit2003} along with Expression (\ref{proofassA_nrc2}). Finally, we show that $\langle Z^f\rangle_{t} \rightarrow \infty$ as $t \rightarrow \infty$. We can calculate that 
\begin{eqnarray}
\nonumber \langle Z^f\rangle_{t}  &=& \langle Z^F\rangle_{t}\\ \nonumber &=& v_{F}(t)\\ \label{proof2_nrc2}
&=& (P_{g,h}^W)^{-1}(F(t)) \mathbf{1}_{\{0 < F(t) < 1\}}
\end{eqnarray}
where we use the fact that $Z^f=Z^F$ in the first equality,  Equation (\ref{v_F_def_nrc2}) from Definition \ref{def_explicitsolution_F_nrc2} in the second equality, and Definition \ref{def_explicitsolution_F_nrc2} in the last equality. By Lemma \ref{lemma_inv_nrc2} we have that $(P_{g,h}^{W})^{-1}(1)=\infty$, and by Definition \ref{Fdef_nrc1} we have that $\underset{t \rightarrow \infty}{\lim} F(t) =1$.  Thus, we can deduce by the assumption that $K_F^1$ is finite from Assumption \ref{assmain_f_nrc2} that 
\begin{eqnarray}
\label{proof3_nrc2}
(P_{g,h}^W)^{-1}(F(t)) \mathbf{1}_{\{0 < F(t) < 1\}} \rightarrow 0
\end{eqnarray}
as $t \rightarrow \infty$. We can deduce by Equations (\ref{proof_nrc2}), (\ref{proof2_nrc2}) and (\ref{proof3_nrc2})
that $\langle Z^f\rangle_{t} du \rightarrow \infty$ as $t \rightarrow \infty$. This implies that $Z^{f}$ satisfies Assumption \ref{assDDS_nrc1}.
\end{proof}
\noindent The next proposition states that if $Z^f$ satisfies Assumption \ref{assDDS_nrc1}, then, the variance function is a solution if and only if it is an explicit solution. The proof is based on substituting the left-hand side of Equation (\ref{PgZf_nrc2}) with Equations (\ref{eq_theorem_nrc2}) from Theorem \ref{theorem_nrc2} and (\ref{sigma_f_def_nrc2}) and then differentiating and inverting on both sides of the equation to derive the explicit solution.
\begin{proposition}
\label{propiii_f_nrc2}
If we assume that $Z^f$ satisfies Assumption \ref{assDDS_nrc1}, then, (i) $\sigma_f^2$ is a solution in the sense of Definition \ref{def_solution_f_nrc2} $\iff$ (ii) $\sigma_f^2$ is an explicit solution in the sense of Definition \ref{def_explicitsolution_f_nrc2}.
\end{proposition}
\begin{proof}[Proof of Proposition \ref{propiii_f_nrc2}] Proof of $(i) \implies (ii)$. We assume that $\sigma_f^2$ is a solution in the sense of Definition \ref{def_solution_f_nrc2}. Given that $Z^f$ satisfies Assumption \ref{assDDS_nrc1}, we can substitute the left-hand side of Equation (\ref{PgZf_nrc2}) with Equation (\ref{eq_theorem_nrc2}) to deduce 
\begin{eqnarray}
\label{proofiii_f_eq0_nrc2}
P_{g,h}^W \Big(\langle Z^f\rangle_{t} \Big) = F(t) \text{ for } t \geq 0. 
\end{eqnarray}
Using Equation (\ref{sigma_f_def_nrc2}), Equation (\ref{proofiii_f_eq0_nrc2}) can be reexpressed as 
\begin{eqnarray}
\label{proofiii_f_eq1_nrc2}
P_{g,h}^W \Big(\int_0^t \sigma_{s,f}^2 ds \Big) = F(t) \text{ for } t \geq 0. 
\end{eqnarray}
By Lemma \ref{lemma_inv_nrc2}, there exists an invert $(P_{g,h}^W)^{-1} : [0,1) \rightarrow \reels^+$.
Applying $(P_{g,h}^W)^{-1}$ on both sides of Equation (\ref{proofiii_f_eq1_nrc2}), Equation (\ref{proofiii_f_eq1_nrc2}) can be rewritten as
\begin{eqnarray}
\label{proofiii_f_eq2_nrc2}
\int_0^t \sigma_{s,f}^2 ds  = & (P_{g,h}^W)^{-1}(F(t)) \mathbf{1}_{\{0 < F(t) < 1\}} & \text{ for } t \geq 0.
\end{eqnarray}
The left-hand side of Equation (\ref{proofiii_f_eq2_nrc2}) and $F$ have a derivative almost everywhere for any $t \geq 0$ by absolute continuity properties and since $F$ is absolutely continuous. $(P_{g,h}^W)^{-1}$ is differentiable on $[0,1)$ by Lemma \ref{lemma_PgWinvertC1_nrc2}. Thus, we can differentiate Equation (\ref{proofiii_f_eq2_nrc2}) almost everywhere on both sides, by using the chain rule on the right-hand side. We obtain
\begin{eqnarray}
\label{sigmatff_nrc2}
\sigma_{t,f}^2 = & f(t)((P_{g,h}^W)^{-1})'(F(t)) \mathbf{1}_{\{0 < F(t) < 1\}}& \text{almost everywhere for } t \geq 0.
\end{eqnarray}
Applying the inverse function theorem, Equation (\ref{sigmatff_nrc2}) can be reexpressed as
\begin{eqnarray*}
\sigma_{t,f}^2 = &\frac{f(t)}{(P_{g,h}^W)'((P_{g,h}^W)^{-1}(F(t)))} \mathbf{1}_{\{0 < F(t) < 1\}} & \text{ almost everywhere for } t \geq 0,
\end{eqnarray*}
or equivalently of the form (\ref{def_expsol_f_nrc2}) as $(P_{g,h}^W)'(t)=f_{g,h}(t)$ almost everywhere for any $t \geq 0$. Thus, we have shown that $\sigma_f^2$ is an explicit solution in the sense of Definition \ref{def_explicitsolution_f_nrc2}.

\noindent Proof of $(ii) \implies (i)$. We assume that $\sigma_f^2$ is an explicit solution in the sense of Definition \ref{def_explicitsolution_f_nrc2}. We have almost everywhere for $t \geq 0$
\begin{eqnarray*}
P_{g,h}^{Z^f} (t)  & = & P_{g,h}^W (\int_0^t \sigma_{s,f}^2 ds)\\
 & = & P_{g,h}^W (\int_0^t \frac{f(s)}{f_{g,h}^W((P_{g,h}^W)^{-1}(F(s)))} \mathbf{1}_{\{0 < F(t) < 1\}} ds)\\ 
& = & P_{g,h}^W (\int_0^t f(s)((P_{g,h}^W)^{-1})'(F(s)) \mathbf{1}_{\{0 < F(t) < 1\}} ds)\\
& = & P_{g,h}^W ((P_{g,h}^W)^{-1})(F(t))\\
& = & F(t),
\end{eqnarray*}
where we use Equation (\ref{eq_theorem_nrc2}) along with the assumption that $Z^f$ satisfies Assumption \ref{assDDS_nrc1} in the first equality, Equation (\ref{def_expsol_f_nrc2}) in the second equality, the inverse function theorem in the third equality, integration in the fourth equality and algebraic manipulation in the fifth equality.
We have thus shown that $\sigma_t^2$ satisfies Equation (\ref{tildesigma_f_def_nrc2}), and thus that $\sigma_f^2$ is a solution in the sense of Definition \ref{def_solution_f_nrc2}.
\end{proof}
\noindent The proof of Theorem \ref{theorem_f_nrc2} is a direct application of Proposition \ref{prop_ass_f_nrc2} and Proposition \ref{propiii_f_nrc2}.
\begin{proof}[Proof of Theorem \ref{theorem_f_nrc2}] To obtain (a), we apply Proposition \ref{prop_ass_f_nrc2} along with Assumption \ref{assmain_f_nrc2}. Then, an application of Proposition \ref{propiii_f_nrc2} along with (a) yields (b). 
\end{proof}

\subsection{One-sided (ii) case}
The next proposition shows that Assumption \ref{assmain_F_rc1} implies that $Y^{F}$ satisfies Assumption \ref{assDDS_rc1}.
\begin{proposition}
\label{prop_ass_F_rc1}
Under Assumption \ref{assmain_F_rc1}, we have that $Y^{F}$ satisfies Assumption \ref{assDDS_rc1}.
\end{proposition}

\begin{proof}[Proof of Proposition \ref{prop_ass_F_rc1}] By Definition \ref{def_explicitsolution_F_rc1}, $Y^F$ is defined as a continuous local martingale with quadratic variation $\langle Y^F\rangle_{t} (\omega) = v_{F}(t,\omega) \text{ for } t \geq 0$ and $\omega \in \Omega$, which can be expressed by Equation (\ref{expsoldefsimp_rc1}) as 
\begin{eqnarray*}
v_{F}(t,\omega) = & \frac{1}{2\operatorname{erfinv}(1-F(t,\omega))^2} \mathbf{1}_{\{0 < F(t,\omega) < 1\}} & \text{ for } t \geq 0 \text{ and } \omega \in \Omega.
\end{eqnarray*}
By definition we have that $\operatorname{erfinv}(z) \rightarrow 0$ as $z\rightarrow 0$, and by Definition \ref{Fdef_rc1} we have that $\underset{t \rightarrow \infty}{\lim} F(t,\omega) =1$. Thus, we can deduce by Assumption \ref{assmain_F_rc1} that $\underset{t \rightarrow \infty}{\lim} v_F(t,\omega) =\infty$. This implies that $\langle Y^F \rangle_{\infty} = \infty$ and thus that $Y^{F}$ satisfies Assumption \ref{assDDS_rc1}.
\end{proof}
\noindent The next proposition states that if a nondecreasing function satisfies Assumption \ref{assDDS_rc1}, then it is a solution if and only if it is an explicit solution. The proof is based on substituting the left-hand side of Equation (\ref{PgZF_rc1}) with Equations (\ref{eq_theorem_rc1}) and (\ref{v_F_def_rc1}) and then inverting on both sides of the equation to derive the explicit solution.
\begin{proposition}
\label{propiii_F_rc1}
If we assume that $v_F$ satisfies Assumption \ref{assDDS_rc1}, then, (i) $v_F$ is a solution in the sense of Definition \ref{def_solution_F_rc1} $\iff$ (ii) $v_F$ is an explicit solution in the sense of Definition \ref{def_explicitsolution_F_rc1}.
\end{proposition}
\begin{proof}[Proof of Proposition \ref{propiii_F_rc1}] Proof of $(i) \implies (ii)$. We assume that $v_F$ is a solution in the sense of Definition \ref{def_solution_F_rc1}. Given that $Y^F$ satisfies Assumption \ref{assDDS_rc1}, we can substitute the left-hand side of Equation (\ref{PgZF_rc1}) with Equation (\ref{eq_theorem_rc1}) to deduce 
\begin{eqnarray}
\label{proofiii_F_eq0_rc1}
P_1^W \Big(\langle Y^F\rangle_{t}(\omega) \Big) = F(t,\omega) \text{ for } t \geq 0 \text{ and } \omega \in \Omega. 
\end{eqnarray}
Using Equation (\ref{v_F_def_rc1}), Equation (\ref{proofiii_F_eq0_rc1}) can be reexpressed as 
\begin{eqnarray}
\label{proofiii_F_eq1_rc1}
P_1^W \Big(v_F(t,\omega) \Big) = F(t,\omega) \text{ for } t \geq 0 \text{ and } \omega \in \Omega. 
\end{eqnarray}
By Lemma \ref{lemma_inv_nrc1}, there exists an invert $(P_1^W)^{-1} : [0,1) \rightarrow \reels^+$.
Applying $(P_1^W)^{-1}$ on both sides of Equation (\ref{proofiii_F_eq1_rc1}), Equation (\ref{proofiii_F_eq1_rc1}) can be rewritten as
Equation (\ref{def_expsol_F_rc1}).

\noindent Proof of $(ii) \implies (i)$. We assume that $v_F$ is an explicit solution in the sense of Definition \ref{def_explicitsolution_F_rc1}. We have
\begin{eqnarray*}
P_1^{Y^F} (t | \omega)  & = & P_1^W \Big(\langle Y^F\rangle_{t}(\omega) \Big) \\
 & = & P_1^W \Big(v_F(t,\omega) \Big)\\
& = & P_1^W \Big((P_1^W)^{-1}(F(t,\omega)) \mathbf{1}_{\{0 < F(t,\omega) < 1\}})\Big)\\
& = & F(t,\omega),
\end{eqnarray*}
where we use Equation (\ref{eq_theorem_rc1}) along with the assumption that $v_F$ satisfies Assumption \ref{assDDS_rc1} in the first equality, Equation (\ref{v_F_def_rc1}) in the second equality, Equation (\ref{def_expsol_F_rc1}) in the third equality, and algebraic manipulation in the fourth equality.
\end{proof}
\noindent The proof of Theorem \ref{theorem_F_rc1} is a direct application of Proposition \ref{prop_ass_F_rc1} and Proposition \ref{propiii_F_rc1}.
\begin{proof}[Proof of Theorem \ref{theorem_F_rc1}] To obtain (a), we apply Proposition \ref{prop_ass_F_rc1} along with Assumption \ref{assmain_F_rc1}. Then, an application of Proposition \ref{propiii_F_rc1} along with (a) yields (b). 
\end{proof}

\noindent The next proposition shows that Assumption \ref{assmain_f_rc1} implies that $Y^{f}$ satisfies Assumption \ref{assDDS_rc1}. The proof is mainly based on the use of Assumption \ref{assmain_f_rc1}. 
\begin{proposition}
\label{prop_ass_f_rc1}
Under Assumption \ref{assmain_f_rc1}, we have that $Y^{f}$ satisfies Assumption \ref{assDDS_rc1}.
\end{proposition}

\begin{proof}[Proof of Proposition \ref{prop_ass_f_rc1}] 
We can deduce that $Y^f$ is a local martingale with random quadratic variation 
\begin{eqnarray}
\label{proof_rc1}
\langle Y^f\rangle_{t} (\omega)  &=& \int_0^t \sigma_{u,f}^2(\omega) du \text{ for } t \geq 0 \text{ and } \omega \in \Omega
\end{eqnarray} 
by Theorem I.4.40 (p. 48) from \cite{JacodLimit2003} along with Expression (\ref{assumptionvol1_rc1}) from Assumption \ref{assmain_f_rc1}. We show that $\langle Y^f\rangle_{t} \rightarrow \infty$ as $t \rightarrow \infty$. We can calculate that 
\begin{eqnarray}
\nonumber \langle Y^f\rangle_{t}(\omega)  &=& \langle Y^F\rangle_{t}(\omega)\\ \nonumber &=& v_{F}(t,\omega)\\ \label{proof2_rc1}
&=& \frac{1}{2\operatorname{erfinv}(1-F(t,\omega))^2} \mathbf{1}_{\{0 < F(t,\omega) < 1\}} \text{ for } t \geq 0 \text{ and } \omega \in \Omega,
\end{eqnarray}
where we use the fact that $Y^f=Y^F$ in the first equality,  Equation (\ref{v_F_def_rc1}) from Definition \ref{def_explicitsolution_F_rc1} in the second equality, and Equation (\ref{expsoldefsimp_rc1}) in the last equality. By definition we have that $\operatorname{erfinv}(z) \rightarrow 0$ as $z\rightarrow 0$, and by Definition \ref{Fdef_rc1} we have that $\underset{t \rightarrow \infty}{\lim} F(t,\omega) =1$. Thus, we can deduce by the assumption that $K_F^1$ is finite from Assumption \ref{assmain_f_rc1} that 
\begin{eqnarray}
\label{proof3_rc1}
\frac{1}{2\operatorname{erfinv}(1-F(t,\omega))^2} \mathbf{1}_{\{0 < F(t,\omega) < 1\}} \rightarrow 0
\end{eqnarray}
as $t \rightarrow \infty$. We can deduce by Equations (\ref{proof_rc1}), (\ref{proof2_rc1}) and (\ref{proof3_rc1})
that $\langle Y^f\rangle_{t} \rightarrow \infty$ as $t \rightarrow \infty$. This implies that $Y^{f}$ satisfies Assumption \ref{assDDS_rc1}.
\end{proof}
\noindent The next proposition states that if $Y^f$ satisfies Assumption \ref{assDDS_rc1}, then, the variance function is a solution if and only if it is an explicit solution. The proof is based on substituting the left-hand side of Equation (\ref{PgZf_rc1}) with Equations (\ref{eq_theorem_rc1}) from Theorem \ref{theorem_rc1} and (\ref{sigma_f_def_rc1}) and then differentiating and inverting on both sides of the equation to derive the explicit solution.
\begin{proposition}
\label{propiii_f_rc1}
If we assume that $Y^f$ satisfies Assumption \ref{assDDS_rc1}, then, (i) $\sigma_f^2$ is a solution in the sense of Definition \ref{def_solution_f_rc1} $\iff$ (ii) $\sigma_f^2$ is an explicit solution in the sense of Definition \ref{def_explicitsolution_f_rc1}.
\end{proposition}
\begin{proof}[Proof of Proposition \ref{propiii_f_rc1}] Proof of $(i) \implies (ii)$. We assume that $\sigma_f^2$ is a solution in the sense of Definition \ref{def_solution_f_rc1}. Given that $Y^f$ satisfies Assumption \ref{assDDS_rc1}, we can substitute the left-hand side of Equation (\ref{PgZf_rc1}) with Equation (\ref{eq_theorem_rc1}) to deduce 
\begin{eqnarray}
\label{proofiii_f_eq0_rc1}
P_1^W \Big(\langle Z^f\rangle_{t}(\omega) \Big) = F(t,\omega) \text{ for } t \geq 0 \text{ and } \omega \in \Omega. 
\end{eqnarray}
Using Equation (\ref{sigma_f_def_rc1}), Equation (\ref{proofiii_f_eq0_rc1}) can be reexpressed as 
\begin{eqnarray}
\label{proofiii_f_eq1_rc1}
P_1^W \Big(\int_0^t \sigma_{s,f}^2(\omega) ds \Big) = F(t,\omega) \text{ for } t \geq 0 \text{ and } \omega \in \Omega. 
\end{eqnarray}
By Lemma \ref{lemma_inv_nrc1}, there exists an invert $(P_1^W)^{-1} : [0,1) \rightarrow \reels^+$.
Applying $(P_1^W)^{-1}$ on both sides of Equation (\ref{proofiii_f_eq1_rc1}), Equation (\ref{proofiii_f_eq1_rc1}) can be rewritten as
\begin{eqnarray}
\label{proofiii_f_eq2_rc1}
\int_0^t \sigma_{s,f}^2(\omega) ds  = & (P_1^W)^{-1}(F(t,\omega)) \mathbf{1}_{\{0 < F(t,\omega) < 1\}} & \text{ for } t \geq 0 \text{ and } \omega \in \Omega.
\end{eqnarray}
The left-hand side of Equation (\ref{proofiii_f_eq2_rc1}) and $F$ have a derivative almost everywhere for any $t \geq 0$ by absolute continuity properties and since $F$ is absolutely continuous. $(P_1^W)^{-1}$ is differentiable on $[0,1)$ by Lemma \ref{lemma_PgWinvertC1_nrc1}. Thus, we can differentiate Equation (\ref{proofiii_f_eq2_rc1}) almost everywhere on both sides, by using the chain rule on the right-hand side. We obtain
\begin{eqnarray}
\label{sigmatff_rc1}
\sigma_{t,f}^2(\omega) =  f(t,\omega)((P_1^W)^{-1}(\omega))'(F(t,\omega)) \mathbf{1}_{\{0 < F(t,\omega) < 1\}} \\\nonumber \text{almost everywhere for } t \geq 0 \text{ and } \omega \in \Omega.
\end{eqnarray}
Applying the inverse function theorem, Equation (\ref{sigmatff_rc1}) can be reexpressed as
\begin{eqnarray*}
\sigma_{t,f}^2(\omega) = &\frac{f(t,\omega)}{(P_1^W)'((P_1^W)^{-1}(F(t,\omega)))} \mathbf{1}_{\{0 < F(t,\omega) < 1\}} & \text{ almost everywhere for } t \geq 0 \text{ and } \omega \in \Omega,
\end{eqnarray*}
or equivalently of the form (\ref{def_expsol_f_rc1}) as $(P_1^W)'(t)=f_1(t)$ almost everywhere for any $t \geq 0$. Thus, we have shown that $\sigma_f^2$ is an explicit solution in the sense of Definition \ref{def_explicitsolution_f_rc1}.

\noindent Proof of $(ii) \implies (i)$. We assume that $\sigma_f^2$ is an explicit solution in the sense of Definition \ref{def_explicitsolution_f_rc1}. We have almost everywhere for $t \geq 0$ and $\omega \in \Omega$ that
\begin{eqnarray*}
P_1^{Y^f} (t | \omega)  & = & P_1^W (\int_0^t \sigma_{s,f}^2(\omega) ds)\\
 & = & P_1^W (\int_0^t \frac{f(s,\omega)}{f_1^W((P_1^W)^{-1}(F(s,\omega)))} \mathbf{1}_{\{0 < F(s,\omega) < 1\}} ds)\\ 
& = & P_1^W (\int_0^t f(s,\omega)((P_1^W)^{-1})'(F(s,\omega)) \mathbf{1}_{\{0 < F(s,\omega) < 1\}} ds)\\
& = & P_1^W ((P_1^W)^{-1})(F(t,\omega))\\
& = & F(t,\omega).
\end{eqnarray*}
where we use Equation (\ref{eq_theorem_rc1}) along with the assumption that $Y^f$ satisfies Assumption \ref{assDDS_rc1} in the first equality, Equation (\ref{def_expsol_f_rc1}) in the second equality, the inverse function theorem in the third equality, integration in the fourth equality and algebraic manipulation in the fifth equality.
We have thus shown that $\sigma_t^2$ satisfies Equation (\ref{tildesigma_f_def_rc1}), and thus that $\sigma_f^2$ is a solution in the sense of Definition \ref{def_solution_f_rc1}.
\end{proof}
\noindent The proof of Theorem \ref{theorem_f_rc1} is a direct application of Proposition \ref{prop_ass_f_rc1} and Proposition \ref{propiii_f_rc1}.
\begin{proof}[Proof of Theorem \ref{theorem_f_rc1}] To obtain (a), we apply Proposition \ref{prop_ass_f_rc1} along with Assumption \ref{assmain_f_rc1}. Then, an application of Proposition \ref{propiii_f_rc1} along with (a) yields (b). 
\end{proof}

\subsection{Two-sided (ii) case}
The next proposition shows that Assumption \ref{assmain_F_rc1} implies that $Z^{F}$ satisfies Assumption \ref{assDDS_rc1}.
\begin{proposition}
\label{prop_ass_F_rc2}
Under Assumption \ref{assmain_F_rc1}, we have that $Z^{F}$ satisfies Assumption \ref{assDDS_rc1}.
\end{proposition}

\begin{proof}[Proof of Proposition \ref{prop_ass_F_rc2}] By Definition \ref{def_explicitsolution_F_rc2}, $Z^F$ is defined as a continuous local martingale with quadratic variation $\langle Z^F\rangle_{t} (\omega) = v_{F}(t,\omega) \text{ for } t \geq 0$ and $\omega \in \Omega$, which can be expressed as 
\begin{eqnarray*}
v_{F}(t,\omega) = & (P_{g,h}^W)^{-1}(F(t,\omega)) \mathbf{1}_{\{0 < F(t,\omega) < 1\}}& \text{ for } t \geq 0.
\end{eqnarray*}
By Lemma \ref{lemma_inv_nrc2} we have that $(P_{g,h}^{W})^{-1}(1)=\infty$, and by Definition \ref{Fdef_rc1} we have that $\underset{t \rightarrow \infty}{\lim} F(t,\omega) =1$. Thus, we can deduce by Assumption \ref{assmain_F_rc1} that $\underset{t \rightarrow \infty}{\lim} v_F(t,\omega) =\infty$. This implies that $\langle Z^F \rangle_{\infty} = \infty$ and thus that $Z^{F}$ satisfies Assumption \ref{assDDS_rc1}.
\end{proof}
\noindent The next proposition states that if a nondecreasing function satisfies Assumption \ref{assDDS_rc1}, then it is a solution if and only if it is an explicit solution. The proof is based on substituting the left-hand side of Equation (\ref{PgZF_rc2}) with Equations (\ref{eq_theorem_rc2}) and (\ref{v_F_def_rc2}) and then inverting on both sides of the equation to derive the explicit solution.
\begin{proposition}
\label{propiii_F_rc2}
If we assume that $v_F$ satisfies Assumption \ref{assDDS_rc1}, then, (i) $v_F$ is a solution in the sense of Definition \ref{def_solution_F_rc2} $\iff$ (ii) $v_F$ is an explicit solution in the sense of Definition \ref{def_explicitsolution_F_rc2}.
\end{proposition}
\begin{proof}[Proof of Proposition \ref{propiii_F_rc2}] Proof of $(i) \implies (ii)$. We assume that $v_F$ is a solution in the sense of Definition \ref{def_solution_F_rc2}. Given that $Z^F$ satisfies Assumption \ref{assDDS_rc1}, we can substitute the left-hand side of Equation (\ref{PgZF_rc2}) with Equation (\ref{eq_theorem_rc2}) to deduce 
\begin{eqnarray}
\label{proofiii_F_eq0_rc2}
P_{g,h}^W \Big(\langle Z^F\rangle_{t}(\omega) \Big) = F(t,\omega) \text{ for } t \geq 0 \text{ and } \omega \in \Omega. 
\end{eqnarray}
Using Equation (\ref{v_F_def_rc2}), Equation (\ref{proofiii_F_eq0_rc2}) can be reexpressed as 
\begin{eqnarray}
\label{proofiii_F_eq1_rc2}
P_{g,h}^W \Big(v_F(t,\omega) \Big) = F(t,\omega) \text{ for } t \geq 0 \text{ and } \omega \in \Omega. 
\end{eqnarray}
By Lemma \ref{lemma_inv_nrc1}, there exists an invert $(P_{g,h}^W)^{-1} : [0,1) \rightarrow \reels^+$.
Applying $(P_{g,h}^W)^{-1}$ on both sides of Equation (\ref{proofiii_F_eq1_rc2}), Equation (\ref{proofiii_F_eq1_rc2}) can be rewritten as
Equation (\ref{def_expsol_F_rc2}).

\noindent Proof of $(ii) \implies (i)$. We assume that $v_F$ is an explicit solution in the sense of Definition \ref{def_explicitsolution_F_rc2}. We have
\begin{eqnarray*}
P_{g,h}^{Z^F} (t | \omega)  & = & P_{g,h}^W \Big(\langle Z^F\rangle_{t}(\omega) \Big) \\
 & = & P_{g,h}^W \Big(v_F(t,\omega) \Big)\\
& = & P_{g,h}^W \Big((P_{g,h}^W)^{-1}(F(t,\omega)) \mathbf{1}_{\{0 < F(t,\omega) < 1\}})\Big)\\
& = & F(t,\omega),
\end{eqnarray*}
where we use Equation (\ref{eq_theorem_rc2}) along with the assumption that $v_F$ satisfies Assumption \ref{assDDS_rc1} in the first equality, Equation (\ref{v_F_def_rc2}) in the second equality, Equation (\ref{def_expsol_F_rc2}) in the third equality, and algebraic manipulation in the fourth equality.
\end{proof}
\noindent The proof of Theorem \ref{theorem_F_rc2} is a direct application of Proposition \ref{prop_ass_F_rc2} and Proposition \ref{propiii_F_rc2}.
\begin{proof}[Proof of Theorem \ref{theorem_F_rc2}] To obtain (a), we apply Proposition \ref{prop_ass_F_rc2} along with Assumption \ref{assmain_F_rc1}. Then, an application of Proposition \ref{propiii_F_rc2} along with (a) yields (b). 
\end{proof}

\noindent The next proposition shows that Assumption \ref{assmain_f_rc1} implies that $Z^{f}$ satisfies Assumption \ref{assDDS_rc1}. The proof is mainly based on the use of Assumption \ref{assmain_f_rc1}. 
\begin{proposition}
\label{prop_ass_f_rc2}
Under Assumption \ref{assmain_f_rc1}, we have that $Z^{f}$ satisfies Assumption \ref{assDDS_rc1}.
\end{proposition}

\begin{proof}[Proof of Proposition \ref{prop_ass_f_rc2}] 
We can deduce that $Z^f$ is a local martingale with random quadratic variation 
\begin{eqnarray}
\label{proof_rc2}
\langle Z^f\rangle_{t} (\omega)  &=& \int_0^t \sigma_{u,f}^2(\omega) du \text{ for } t \geq 0 \text{ and } \omega \in \Omega
\end{eqnarray} 
by Theorem I.4.40 (p. 48) from \cite{JacodLimit2003} along with Expression (\ref{assumptionvol1_rc1}) from Assumption \ref{assmain_f_rc1}. We show that $\langle Z^f\rangle_{t} \rightarrow \infty$ as $t \rightarrow \infty$. We can calculate that 
\begin{eqnarray}
\nonumber \langle Z^f\rangle_{t}(\omega)  &=& \langle Z^F\rangle_{t}(\omega)\\ \nonumber &=& v_{F}(t,\omega)\\ \label{proof2_rc2}
&=& (P_{g,h}^W)^{-1}(F(t,\omega)) \mathbf{1}_{\{0 < F(t,\omega) < 1\}} \text{ for } t \geq 0 \text{ and } \omega \in \Omega,
\end{eqnarray}
where we use the fact that $Z^f=Z^F$ in the first equality,  Equation (\ref{v_F_def_rc1}) from Definition \ref{def_explicitsolution_F_rc1} in the second equality, and Equation (\ref{def_expsol_F_rc2}) in the last equality. By Lemma \ref{lemma_inv_nrc2} we have that $(P_{g,h}^{W})^{-1}(1)=\infty$, and by Definition \ref{Fdef_rc1} we have that $\underset{t \rightarrow \infty}{\lim} F(t,\omega) =1$. Thus, we can deduce by the assumption that $K_F^1$ is finite from Assumption \ref{assmain_f_rc1} that 
\begin{eqnarray}
\label{proof3_rc2}
(P_{g,h}^W)^{-1}(F(t,\omega))\mathbf{1}_{\{0 < F(t,\omega) < 1\}} \rightarrow 0
\end{eqnarray}
as $t \rightarrow \infty$. We can deduce by Equations (\ref{proof_rc2}), (\ref{proof2_rc2}) and (\ref{proof3_rc2})
that $\langle Z^f\rangle_{t} \rightarrow \infty$ as $t \rightarrow \infty$. This implies that $Z^{f}$ satisfies Assumption \ref{assDDS_rc1}.
\end{proof}
\noindent The next proposition states that if $Z^f$ satisfies Assumption \ref{assDDS_rc1}, then, the variance function is a solution if and only if it is an explicit solution. The proof is based on substituting the left-hand side of Equation (\ref{PgZf_rc2}) with Equations (\ref{eq_theorem_rc2}) from Theorem \ref{theorem_rc2} and (\ref{sigma_f_def_rc2}) and then differentiating and inverting on both sides of the equation to derive the explicit solution.
\begin{proposition}
\label{propiii_f_rc2}
If we assume that $Z^f$ satisfies Assumption \ref{assDDS_rc1}, then, (i) $\sigma_f^2$ is a solution in the sense of Definition \ref{def_solution_f_rc2} $\iff$ (ii) $\sigma_f^2$ is an explicit solution in the sense of Definition \ref{def_explicitsolution_f_rc2}.
\end{proposition}
\begin{proof}[Proof of Proposition \ref{propiii_f_rc2}] Proof of $(i) \implies (ii)$. We assume that $\sigma_f^2$ is a solution in the sense of Definition \ref{def_solution_f_rc2}. Given that $Z^f$ satisfies Assumption \ref{assDDS_rc1}, we can substitute the left-hand side of Equation (\ref{PgZf_rc2}) with Equation (\ref{eq_theorem_rc2}) to deduce 
\begin{eqnarray}
\label{proofiii_f_eq0_rc2}
P_{g,h}^W \Big(\langle Z^f\rangle_{t}(\omega) \Big) = F(t,\omega) \text{ for } t \geq 0 \text{ and } \omega \in \Omega. 
\end{eqnarray}
Using Equation (\ref{sigma_f_def_rc2}), Equation (\ref{proofiii_f_eq0_rc2}) can be reexpressed as 
\begin{eqnarray}
\label{proofiii_f_eq1_rc2}
P_{g,h}^W \Big(\int_0^t \sigma_{s,f}^2(\omega) ds \Big) = F(t,\omega) \text{ for } t \geq 0 \text{ and } \omega \in \Omega. 
\end{eqnarray}
By Lemma \ref{lemma_inv_nrc1}, there exists an invert $(P_{g,h}^W)^{-1} : [0,1) \rightarrow \reels^+$.
Applying $(P_{g,h}^W)^{-1}$ on both sides of Equation (\ref{proofiii_f_eq1_rc2}), Equation (\ref{proofiii_f_eq1_rc2}) can be rewritten as
\begin{eqnarray}
\label{proofiii_f_eq2_rc2}
\int_0^t \sigma_{s,f}^2(\omega) ds  = & (P_{g,h}^W)^{-1}(F(t,\omega)) \mathbf{1}_{\{0 < F(t,\omega) < 1\}} & \text{ for } t \geq 0 \text{ and } \omega \in \Omega.
\end{eqnarray}
The left-hand side of Equation (\ref{proofiii_f_eq2_rc2}) and $F$ have a derivative almost everywhere for any $t \geq 0$ by absolute continuity properties and since $F$ is absolutely continuous. $(P_{g,h}^W)^{-1}$ is differentiable on $[0,1)$ by Lemma \ref{lemma_PgWinvertC1_nrc1}. Thus, we can differentiate Equation (\ref{proofiii_f_eq2_rc2}) almost everywhere on both sides, by using the chain rule on the right-hand side. We obtain
\begin{eqnarray}
\label{sigmatff_rc2}
\sigma_{t,f}^2(\omega) =  f(t,\omega)((P_{g,h}^W)^{-1}(\omega))'(F(t,\omega)) \mathbf{1}_{\{0 < F(t,\omega) < 1\}} \\\nonumber \text{almost everywhere for } t \geq 0 \text{ and } \omega \in \Omega.
\end{eqnarray}
Applying the inverse function theorem, Equation (\ref{sigmatff_rc2}) can be reexpressed as
\begin{eqnarray*}
\sigma_{t,f}^2(\omega) = &\frac{f(t,\omega)}{(P_{g,h}^W)'((P_{g,h}^W)^{-1}(F(t,\omega)))} \mathbf{1}_{\{0 < F(t,\omega) < 1\}} & \text{ almost everywhere for } t \geq 0 \text{ and } \omega \in \Omega,
\end{eqnarray*}
or equivalently of the form (\ref{def_expsol_f_rc1}) as $(P_{g,h}^W)'(t)=f_{g,h}(t)$ almost everywhere for any $t \geq 0$. Thus, we have shown that $\sigma_f^2$ is an explicit solution in the sense of Definition \ref{def_explicitsolution_f_rc2}.

\noindent Proof of $(ii) \implies (i)$. We assume that $\sigma_f^2$ is an explicit solution in the sense of Definition \ref{def_explicitsolution_f_rc2}. We have almost everywhere for $t \geq 0$ and $\omega \in \Omega$ that
\begin{eqnarray*}
P_{g,h}^{Z^f} (t | \omega)  & = & P_{g,h}^W (\int_0^t \sigma_{s,f}^2(\omega) ds)\\
 & = & P_{g,h}^W (\int_0^t \frac{f(s,\omega)}{f_{g,h}^W((P_{g,h}^W)^{-1}(F(s,\omega)))} \mathbf{1}_{\{0 < F(s,\omega) < 1\}} ds)\\ 
& = & P_{g,h}^W (\int_0^t f(s)((P_{g,h}^W)^{-1})'(F(s,\omega)) \mathbf{1}_{\{0 < F(s,\omega) < 1\}} ds)\\
& = & P_{g,h}^W ((P_{g,h}^W)^{-1})(F(t,\omega))\\
& = & F(t,\omega).
\end{eqnarray*}
where we use Equation (\ref{eq_theorem_rc2}) along with the assumption that $Z^f$ satisfies Assumption \ref{assDDS_rc1} in the first equality, Equation (\ref{def_expsol_f_rc2}) in the second equality, the inverse function theorem in the third equality, integration in the fourth equality and algebraic manipulation in the fifth equality.
We have thus shown that $\sigma_t^2$ satisfies Equation (\ref{tildesigma_f_def_rc2}), and thus that $\sigma_f^2$ is a solution in the sense of Definition \ref{def_solution_f_rc2}.
\end{proof}
\noindent The proof of Theorem \ref{theorem_f_rc2} is a direct application of Proposition \ref{prop_ass_f_rc2} and Proposition \ref{propiii_f_rc2}.
\begin{proof}[Proof of Theorem \ref{theorem_f_rc2}] To obtain (a), we apply Proposition \ref{prop_ass_f_rc2} along with Assumption \ref{assmain_f_rc1}. Then, an application of Proposition \ref{propiii_f_rc2} along with (a) yields (b). 
\end{proof}

%\section{Conclusion}

%%%%%%%%%%%%%%%%%%%%%%%%%%%%%%%%%%%%%%%%%%%%%%
%% Example with single Appendix:            %%
%%%%%%%%%%%%%%%%%%%%%%%%%%%%%%%%%%%%%%%%%%%%%%
%\begin{appendix}

%\section{Proofs}

%\end{appendix}

%%%%%%%%%%%%%%%%%%%%%%%%%%%%%%%%%%%%%%%%%%%%%%
%% Support information, if any,             %%
%% should be provided in the                %%
%% Acknowledgements section.                %%
%%%%%%%%%%%%%%%%%%%%%%%%%%%%%%%%%%%%%%%%%%%%%%
\newpage 
%\begin{acks}[Acknowledgments]
%\end{acks}

%%%%%%%%%%%%%%%%%%%%%%%%%%%%%%%%%%%%%%%%%%%%%%
%% Funding information, if any,             %%
%% should be provided in the                %%
%% funding section.                         %%
%%%%%%%%%%%%%%%%%%%%%%%%%%%%%%%%%%%%%%%%%%%%%%
\begin{funding}
The author was supported in part by Japanese Society for the Promotion of Science Grants-in-Aid for Scientific Research (B) 23H00807 and Early-Career Scientists 20K13470. 
\end{funding}

%%%%%%%%%%%%%%%%%%%%%%%%%%%%%%%%%%%%%%%%%%%%%%
%% Supplementary Material, including data   %%
%% sets and code, should be provided in     %%
%% {supplement} environment with title      %%
%% and short description. It cannot be      %%
%% available exclusively as external link.  %%
%% All Supplementary Material must be       %%
%% available to the reader on Project       %%
%% Euclid with the published article.       %%
%%%%%%%%%%%%%%%%%%%%%%%%%%%%%%%%%%%%%%%%%%%%%%
%\begin{supplement}
%\stitle{Title of Supplement A}
%\sdescription{Short description of Supplement A.}
%\end{supplement}

%%  MR numbers will be added by VTeX.                      %%
%%                                                         %%
%%  Use \cite{...} to cite references in text.             %%
%%                                                         %%
%%%%%%%%%%%%%%%%%%%%%%%%%%%%%%%%%%%%%%%%%%%%%%%%%%%%%%%%%%%%%

%% if your bibliography is in bibtex format, uncomment commands:
\bibliographystyle{imsart-nameyear} % Style BST file (imsart-number.bst or imsart-nameyear.bst)
\bibliography{biblio_abb}       % Bibliography file (usually '*.bib')

\end{document}